\documentclass{amsart}

\usepackage{amssymb, amscd}
\usepackage{stmaryrd}
\usepackage{color,MnSymbol}

\usepackage[initials]{amsrefs}

\usepackage{tikz}
\usetikzlibrary{decorations.pathmorphing}
\usetikzlibrary{arrows}

\newcommand{\si}{\sigma}
\newcommand{\Si}{\Sigma}
\newcommand{\Ga}{\Gamma}

\newcommand{\bC}{\mathbb{C}}

\newcommand{\bP}{\mathbb{P}}

\newcommand{\bR}{\mathbb{R}}

\newcommand{\bZ}{\mathbb{Z}}
\newcommand{\btau}{\boldsymbol{\tau}}
\newcommand{\bsi}{{\boldsymbol{\si}}}

\newcommand{\cE}{\mathcal{E}}
\newcommand{\cL}{\mathcal{L}}

\newcommand{\cM}{\mathcal{M}}
\newcommand{\cO}{\mathcal{O}}

\newcommand{\cS}{\mathcal{S}}

\newcommand{\cT}{\mathcal{T}}
\newcommand{\cU}{\mathcal{U}}

\newcommand{\cX}{\mathcal{X}}
\newcommand{\cY}{\mathcal{Y}}
\newcommand{\cZ}{\mathcal{Z}}

\newcommand{\fp}{\mathfrak{p}}

\newcommand{\fQ}{\mathfrak{Q}}
\newcommand{\fX}{\mathfrak{X}}

\newcommand{\End}{\mathrm{End}}
\newcommand{\Spec}{\mathrm{Spec}}

\newcommand{{\inv} }{\mathrm{inv}}
\newcommand{\ev}{\mathrm{ev}}
\newcommand{\Aut}{\mathrm{Aut}}
\newcommand{\Def}{\mathrm{Def}}
\newcommand{\Res}{\mathrm{Res}}

\newcommand{\val}{ {\mathrm{val}} }
\newcommand{\vir}{{\mathrm{vir}}}

\newcommand{\pt}{ {\mathrm{pt}}}
\newcommand{\Pic}{ {\mathrm{Pic}}}

\newcommand{\one}{\mathbf{1}}

\newcommand{\bu}{\mathbf{u}}

\newcommand{\bp}{\mathbf{p}}

\newcommand{\bGa}{\mathbf{\Ga}}

\newcommand{\sv}{\mathsf{v}}
\newcommand{\sw}{\mathsf{w}}
\newcommand{\sH}{{\mathsf{H}}}

\newcommand{\tF}{ {\widetilde{F}} }

\newcommand{\tfQ}{ \widetilde{\mathfrak{Q}}}

\newcommand{\txi}{ {\widetilde{\xi}} }

\newcommand{\tS}{\widetilde{S}}

\newcommand{\hxi}{\hat{\xi}}

\newcommand{\htheta}{\hat{\theta}}
\newcommand{\that}{\hat{t}}
\newcommand{\phihat}{\hat{\phi}}

\newcommand{\hualpha}{{\hat{\underline \alpha}}}
\newcommand{\hubeta}{{\hat{\underline \beta}}}
\newcommand{\hugamma}{{\hat{\underline \gamma}}}

\newcommand{\vGa}{\vec{\Ga}}
\newcommand{\vmu}{\vec{\mu}}

\newcommand{\Mbar}{\overline{\cM}}

\renewcommand{\L}{{L_{P,Q}}}
\renewcommand{\H}{{\mathsf {H}}}
\newcommand{\T}{{T_{P,Q}}}
\newcommand{\HT}{{\H^{\T}}}

\newcommand{\Rtt}{ {R_{\T}} }
\newcommand{\bStt}{ {\bar{S}_{\T}} }

\newcommand{\nov}{\Lambda_{\mathrm{nov}}}

\newcommand{\novtt}{\bar{\Lambda}^{\T}_{\mathrm{nov}} }

\newtheorem*{theorem*}{Theorem}
\newtheorem{dummy}{dummy}[section]

\newtheorem{theorem}[dummy]{Theorem}

\newtheorem{proposition}[dummy]{Proposition}

\newtheorem{definition}[dummy]{Definition}
\newtheorem{remark}[dummy]{Remark}
\newtheorem*{remark*}{Remark}

\begin{document}

\title[Topological recursion for torus knots]{Topological recursion for the conifold transition of a torus knot}

\begin{abstract}
In this paper we prove a mirror symmetry conjecture based on the work of Brini-Eynard-Mari\~no
\cite{BEM} and Diaconescu-Shende-Vafa \cite{DSV}. This conjecture relates open Gromov-Witten invariants of the conifold
transition of a torus knot to the topological recursion on the B-model
spectral curve.
\end{abstract}

\author{Bohan Fang}
\address{Bohan Fang, Beijing International Center for Mathematical
  Research, Peking University, 5 Yiheyuan Road, Beijing 100871, China}
\email{bohanfang@gmail.com}

\author{Zhengyu Zong}
\address{Zhengyu Zong, Yau Mathematical Sciences Center, Tsinghua University, Jin Chun Yuan West Building,
Tsinghua University, Haidian District, Beijing 100084, China}
\email{zyzong@mail.tsinghua.edu.cn}

\maketitle

\tableofcontents

\section{Introduction}

The Chern-Simons theory of a knot in $S^3$ \cite{Wi89} is related to topological
strings in $T^*S^3$ and, through a conifold transition, to topological
strings in the resolved conifold $\cX=\cO_{\bP^1}(-1)\oplus
\cO_{\bP^1}(-1)$ \cite{Wi95, GV, OV}. We briefly review the related
background and state our result in this section.\footnote{For a comprehensive review, see e.g. \cite{Ma05, Ma10}.}

Let $A$ be a connection on $S^3$ for the gauge group $G=U(N)$ and $R$
be a representation of $G$. The Chern-Simons
action functional is the following (where $k$ is the coupling constant)
\[
S=\frac{k}{4\pi} \int_{S^3} \mathrm{Tr}_R\left( A\wedge d A +\frac{2}{3}
  A\wedge A \wedge A\right ).
\]
The partition function of this theory is defined by path-integrals in physics
\[
Z(S^3)={\int \mathcal D Ae^{\sqrt{-1}S(A)}}.
\]
Let $K\cong S^1\hookrightarrow S^3$ be a framed, oriented knot\footnote{The construction also
  works for links -- we restrict ourselves to knots in this paper.}. In
physics, the normalized vacuum
expectation value (vev) is
\[
W_R(K)=\frac{1}{Z(S^3)}\int \mathcal D A e^{\sqrt{-1}S(A)}\mathrm{Tr}_RU_K,
\]
where $U_K$ is the holonomy around $K$. This definition also relies on
path-integral. In mathematics, for example, when $R$ is the
fundamental representation of $G=U(N)$, $W_R(K)$ is related to the
HOMFLY polynomial $P_K(q,\lambda)$ of $K$ as below
\[
W_R(K)=\lambda
\frac{\lambda^{\frac{1}{2}}-\lambda^{-\frac{1}{2}}}{q^{\frac{1}{2}}-q^{-\frac{1}{2}}}P_K(q,\lambda),
\]
where
\[
q=e^{\frac{2\pi\sqrt{-1}}{k+N}},\quad \lambda=q^N.
\]

Under the large-$N$ duality and the conifold transition, the gauge
theory invariant $W_R(K)$ is conjecturally related to the open Gromov-Witten theory
of $\cX=\cO_{\bP^1}(-1)\oplus \cO_{\bP^1}(-1)$ \cite{GV, OV} with Lagrangian boundary condition $L_K$. When
$K$ is an unknot, the conjecture can be precisely formulated as the famous
Mari\~no-Vafa formula concerning Hodge integrals and is later proved
\cite{MaVa, LLZ, OkPa}.

Although the general construction of $L_K$ out of $K$ other than unknots was not very clear in
the beginning, later the construction \cite{Ta01, Ko07, LMV}
provided recipes for $L_K$. In this paper, we follow
Diaconescu-Shende-Vafa's construction \cite{DSV} for an algebraic knot $K$, which
produces a Maslov index $0$ non-compact Lagrangian $L_K\cong S^1\times
\bR^2$ in $\cX$.

Open Gromov-Witten theory is usually difficult to define for $(\cX,L_K)$. When $K$ is an
unknot, $L_K$ belongs to a class of Lagrangians called Aganagic-Vafa
branes \cite{AgVa,AKV} (Harvey-Lawson type Lagrangian). The
open Gromov-Witten invariants for $(\cX,L_K)$ in this situation are
defined in \cite{KaLiu, Liu02}.  When $K$ is a torus knot, i.e. a knot that can be realized on a real torus in $S^3$, one
can still use the localization technique to define such open
Gromov-Witten invariants \cite{DSV}. In \cite{DSV}, Diaconescu-Shende-Vafa also
prove the conjecture on the correspondence between HOMFLY polynomials and open Gromov-Witten
invariants.

The topological A-type string theory on $\cX$ has mirror symmetry. The
B-model is a spectral curve. When the Lagrangian $L_K$ in $\cX$ is an
Aganagic-Vafa brane, i.e. the conifold transition of an unknot, all
genus Gromov-Witten open-closed invariants are obtained from
Eynard-Orantin's topological recursion of a particular spectral curve \cite{EO07, BKMP, EO15, FLZ16}.
In \cite{BEM}, Brini-Eynard-Mari\~no propose a modified
spectral curve for a torus knot $K_{P,Q}$ with coprime $(P,Q)$. They conjecture that this curve should  be the correct B-model topological strings large-$N$ dual to the knot $K$. Combined with the construction of A-model open Gromov-Witten invariants in \cite{DSV}, one naturally expects the Eynard-Orantin recursion should predict these open GW invariants for $(\cX,L_K)$.

More precisely, we can collect genus $g$, $n$ boundary components open Gromov-Witten invariants and compose them into a generating function $F_{g,n}(X_1,\dots,X_n,\tau_1)$ (see Section \ref{sec:open-closed-GW}). On
the other hand, the spectral curve $C_q$, roughly speaking, is the following curve
\[
C_q=\{1+U+V+q UV=0\},
\]
with a holomorphic function $X=U^Q V^P$ on
$C_q$.

Eynard-Orantin's recursion is a recursive algorithm that produces all genus open invariants of this spectral curve
(see Section \ref{sec:EO-recursion}). From the recursion, we get a symmetric
meromophic $n$-form $\omega_{g,n}$ on $(\overline C_q)^n$. The
variable $\eta=X^{\frac{1}{Q}}$ is a local coordinate around
$(U,V)=(0,-1)$. One can integrate the expansion of
$\omega_{g,n}$ in $\eta_1,\dots,\eta_n$ around this point and define
\[
W_{g,n}(\eta_1,\dots,\eta_n,q)=\int_0^{\eta_1}\dots\int_0^{\eta_n} \omega_{g,n}.
\]
The mirror symmetry for $(\cX,L_K)$ says the following
\begin{theorem*}[Conjecture from Brini-Eynard-Mari\~no, Diaconescu-Shende-Vafa  \cite{BEM,DSV}]
Under $\tau_1=\log q$ and $X_k=\eta_k^Q$, the power series expansion of the open Gromov-Witten amplitude $F_{g,n}(\tau_1;X_1,\dots,X_n)$ in $X_1,\dots,X_n$ is the part in the power
series expansion of $(-1)^{g-1+n} Q^n
W_{g,n}(q,\eta_1,\dots,\eta_n)$ whose degrees of
each $\eta_k$ are divisible by $Q$.
\end{theorem*}

\begin{remark}
In \cite{GJKS}, Gu-Jockers-Klemm-Soroush argue that the B-model
spectral curve of a knot is defined by the \emph{augmentation
  polynomial}. If we choose such augmentation polynomial, although
more complicated than the spectral curve $C_q$ here, we do not need to discard
terms with degrees divisible by $Q$. The authors hope the Main Theorem here will lead to the prediction
by the augmentation variety. An experimental computation of
augmentation polynomial by localization of open GW invariants is in \cite{Mah16}.
\end{remark}

\subsection*{Outline}

In Section \ref{sec:transition} we recall the construction in
\cite{DSV}. Starting from an algebraic knot $K$ in $S^3$ we will
construct a Lagrangian $L_K$ in $\cX$. In Section \ref{sec:A-model} we
define open Gromov-Witten invariants with respect to $(\cX,L_K)$ for a
torus knot $K$ in two ways: by localization in the moludi space of maps from bordered Riemann
surfaces, and by relative Gromov-Witten invariants. We also express
the generating functions for these invariants in graph sums. In
Section \ref{sec:spectral} we discuss the B-model mirror to $(\cX,L_K)$
as a spectral curve $C_q$, and express the Eynard-Orantin invariants in terms of graph
sums. Finally, in Section \ref{sec:mirror} we prove the all genus
mirror symmetry between $(\cX,L_K)$ and $C_q$, based on the localization
computation on disk invariants, genus $0$ mirror theorem and graph sum formulae.

\subsection*{Acknowledgement}
We would like to thank Chiu-Chu Melissa Liu for very helpful
discussion and the wonderful
collaboration in many related projects -- those projects are indispensable to
this one. We also thank her for the construction of disk invariants using relative
Gromov-Witten invariants in our case. The first author would like to
thank Sergei Gukov for the discussion on the localization computation
for torus knots. The work of BF is partially supported by the Recruitment
Program of Global Experts in China and a start-up grant at Peking
University. The work of ZZ is partially supported by the start-up
grant at Tsinghua University.

\section{Torus knots and the resolved conifold}
\label{sec:transition}

In this paper, we consider the \emph{conifold} $\cY_0$ which is a hypersurface in $\bC^4$ defined by the following equation:
\subsection{The conifold transition}
\begin{equation*}
xz-yw=0.
\end{equation*}
Here $x,y,z,w$ are standard coordinates in $\bC^4$. The conifold $\cY_0$ has a singularity at the origin $x=y=z=w=0$. The \emph{deformed conifold} $\cY_\delta$ defined by the following equation:
\begin{equation*}
xz-yw=\delta,
\end{equation*}
where $\delta\in\bC\setminus \{0\}$. Then $\cY_\delta$ is a smooth hypersurface in $\bC^4$. Consider the standard symplectic form on $\bC^4$:
\begin{equation*}
\omega_{\bC^4}=\frac{\sqrt{-1}}{2}(dx\wedge d\bar{x}+dy\wedge d\bar{y}+dz\wedge d\bar{z}+dw\wedge d\bar{w}).
\end{equation*}
The symplectic form on $\cY_\delta$ is $\omega_{Y_\delta}=\omega_{\bC^4}\mid_{Y_\delta}$.
Then $\cY_\delta$ becomes a symplectic manifold and there exists a symplectomorphism $\phi_\delta: \cY_\delta \to T^*S^3$, where $T^*S^3$ is the cotangent bundle of the 3-sphere. Consider the anti-holomorphic involution for $\delta \geq 0$.
\begin{eqnarray}\label{involution}
I:\bC^4&\to& \bC^4\\
(x,y,z,w)&\mapsto& (\bar{z},-\bar{w},\bar{x},-\bar{y}).\nonumber
\end{eqnarray}
Then $\cY_\delta$ is preserved by $I$. When $\delta>0$, the fixed locus $S_\delta$ of the induced anti-holomorphic involution $I_\delta$ on $\cY_\delta$ is isomorphic to a 3-sphere of radius $\sqrt{\delta}$ and $\phi_\delta(S_\delta)$ is the zero section of $T^*S^3$. When $\delta=0$, $S_0$ is the singular point of $\cY_0$. But we still have a symplectomorphism $\phi_0:\cY_0\setminus\{0\}\to T^*S^3\setminus S^3$, where we view $S^3$ as the zero section of $T^*S^3$.

The second way to smooth the singularity of $\cY_0$ is to consider the \emph{resolved conifold} $\cX$. We consider the blow-up of $\bC^4$ along the subspace $\{(x,y,z,w)|y=z=0\}$. Let $\cX$ be the resolution of $\cY_0$ under the blow-up. Then $\cX$ is isomorphic to the local $\bP^1$:
\begin{equation*}
\cX\cong[\cO_{\bP^1}(-1)\oplus\cO_{\bP^1}(-1)\to \bP^1].
\end{equation*}
If we view $\cX$ as a subspace of $\bC^4\times\bP^1$, then $\cX$ is defined by the following equations:
\begin{equation}\label{eqnX}
xs=wt,\quad ys=zt,
\end{equation}
where $[s:t]$ is the homogeneous coordinate on $\bP^1$. The resolution $p:\cX\to \cY_0$ is given by contracting the base $\bP^1$ in $\cX$. We say that $\cX$ and $\cY_\delta$ are related by the \emph{conifold transition}.

\subsection{Torus knots and Lagrangians in the deformed conifold}

\subsubsection{Conormal bundle of a knot in $S^3$}\label{conormal}
Consider a knot $K\subset S^3$. Let $M$ be the total space of the conormal bundle $N^*_K$ of $K$ in $S^3$. Then $M$ can be embedded into $T^*S^3$ as a Lagrangian submanifold and the intersection of $M$ with the zero section of $T^*S^3$ is the knot $K$. Recall that we have a symplectomorphism for $\delta>0$
$$\phi_\delta: \cY_\delta\to T^*S^3$$
and $\phi_\delta(S_\delta)$ is the zero section of $T^*S^3$. So $\phi_\delta^{-1}(M)$ is a Lagrangian submanifold of $Y_\delta$ and the intersection $\phi_\delta^{-1}(M)\cap S_\delta$ is a knot in $S_\delta$ which is isomorphic to $K\subset S^3$.

Our goal is to construct a Lagrangian $L_K$ in the resolved conifold $\cX$, which in some sense corresponds to the knot $K$ under the conifold transition. The difficulty here is that when $\delta\to 0$, the subset $S_\delta\subset \cY_\delta$ shrinks to a point which is the singular point of the conifold $\cY_0$. Since the intersection $\phi_\delta^{-1}(M)\cap S_\delta$ is nonempty, the Lagrangian $\phi_\delta^{-1}(M)$ becomes singular in the limit $\delta\to 0$. So it is not easy to construct a Lagrangian $L_K$ in the resolved conifold $\cX$ which is the ``transition'' of $\phi_\delta^{-1}(M)$.

We follow the solution to the above problem for algebraic knots in
\cite{DSV} (see also \cite{Ko07} for a more general construction), where the knot $K$ is ``lifted''to a path $\gamma$ in $T^*S^3$ which does not intersect with the zero section. The Lagrangian $M$ is also lifted to a new Lagrangian containing $\gamma$ and it does not intersect with the zero section either. Then the ``transition'' $L_K$ of $\phi_\delta^{-1}(M)$ can be naturally constructed. For completeness, we review this process in Section \ref{lift} and Section \ref{brane}.

\subsubsection{Lifting of the torus knots}\label{lift}
We restrict ourselves to torus knots. Let $P,Q$ be two fixed coprime positive integers. Let $f(x,y)=x^P-y^Q$ and consider the algebraic curve
\begin{equation*}
f(x,y)=0
\end{equation*}
in $\bC^2$. For small $r$, the intersection of the curve $f(x,y)=0$ with the 3-sphere $|x|^2+|y|^2=r$ represents a knot in $S^3$ which is called the $(P,Q)$-torus knot. We denote this knot by $K$.

We want to consider the 1-dimensional subvariety $\cZ_\delta\subset \cY_\delta$ defined by the complete intersection of $\cY_\delta$ with
\begin{equation}\label{Z}
f(x,y)=0,\quad f(z,-w)=0.
\end{equation}
The subvariety $\cZ_\delta$ is disconnected in general and its connected components can be described as follows. The Equations \eqref{Z} together with the defining equation of $Y_\delta$ implies
\begin{equation*}
(xz)^P-(\delta-xz)^Q=0.
\end{equation*}
Let $xz=\xi$ and then $\xi$ is a solution to the equation $u^P-(\delta-u)^Q=0$ for $u$. Each solution $\xi$ of this equation determines a connected component of $\cZ_\delta$ of the form
\begin{equation*}
(x,y,z,w)=(t^Q,t^P,\xi t^{-Q},(\delta-\xi)t^{-P})
\end{equation*}
for $t\in\bC^*$.

Since the coefficients of $f$ are real, $\cZ_\delta$ is preserved under the anti-holomorphic involution $I$. Each connected component of the intersection $\cZ_\delta\cap S_\delta$ is isomorphic to the knot $K$ in $S_\delta\cong S^3$. Let
\begin{equation*}
P_a=\{(u,v)\in T^*S^3\mid|v|= a\}
\end{equation*}
be the sphere bundle of radius $a$ in $T^*S^3$ under standard metric of the unit sphere $S^3$. Suppose there exits an irreducible component $C_\delta$ of $\cZ_\delta$ such that the intersection $C_\delta\cap S_\delta$ is isomorphic to the knot $K$ in $S_\delta$. Then for small $a>0$, the intersection $\phi_\delta(C_\delta)\cap P_a$ is nontrivial and the projection $\pi(\phi_\delta(C_\delta)\cap P_a)$ is equal to $\phi_\delta(C_\delta \cap S_\delta) \subset S^3$
which is the torus knot $K$. Here $\pi:T^*S^3 \to S^3$ is the projection map. Let $\gamma_a:S^1\to T^*S^3$ be the path $\phi_\delta(C_\delta)\cap P_a$ and for $t\in S^1$ let $\gamma_a(t)=(g(t),h(t))\in T^*S^3$. Then the path $(g(t),0)\in S^3$ is the knot $K$. The conormal bundle $N^*_K$ is defined as
\begin{equation*}
\{(u,v)\in T^*S^3:u=g(t),\quad\langle v,g'(t)\rangle=0\},
\end{equation*}
where $g'(t)$ is the derivative of $g$ and $\langle ,\rangle$ is the natural pairing between tangent and cotangent vectors. As we discussed in Section \ref{conormal}, the conormal bundle $N^*_K$ is not what we want. Instead, we consider the Lagrangian $M_{\gamma_a}\subset T^*S^3$ defined as
\begin{equation}\label{lag}
\{(u,v)\in T^* S^3:u=g(t),\quad\langle v-h(t),g'(t)\rangle=0\}.
\end{equation}
The Lagrangian $M_{\gamma_a}$ is obtained from $N^*_K$ by fiberwisely translating $N^*_K$ by the cotangent vector $h(t)$. We denote $\phi_\delta^{-1}(M_{\gamma_a})$ by $M_{\delta}$.

When $\delta=0$, $\cZ_0$ has two special irreducible components $C^{\pm}$ defined by
\begin{equation*}
f(x,y)=0,\quad z=w=0
\end{equation*}
and
\begin{equation*}
f(z,-w)=0,\quad x=y=0
\end{equation*}
respectively. Both of $C^{\pm}$ meet the singular point of $\cY_0$ and the anti-holomorphic involution $I_0$ exchanges $C^{\pm}$. Consider the path $\gamma^+$ defined by intersection $\phi_0(C^+\setminus \{0\})\cap P_a$. Then by the construction in Equation (\ref{lag}), we obtain the corresponding Lagrangian $M_{\gamma^+}$ in $T^*S^3$ and we denote $\phi_0^{-1}(M_{\gamma^+})$ by $M_0$. For small $\delta>0$, there exists an irreducible component $C_\delta$ of $\cZ_\delta$ such that there exists a connected component $\gamma_\delta$ of the intersection $\phi_\delta(C_\delta)\cap P_a$ which specializes to $\gamma^+$ when $\delta\to 0$. Therefore we obtain a family of Lagrangians $M_\delta$ which specializes to $M_0$ when $\delta\to 0$.

\subsection{Lagrangians in the resolved conifold}\label{brane}
For $\epsilon\geq 0$, we consider the symplectic form $(\omega_{\bC^4}+\epsilon^2\omega_{\bP^1})$ on $\bC^4\times\bP^1$. By equation (\ref{eqnX}), we can view the resolved conifold $\cX$ as a subvariety in $\bC^4\times\bP^1$. We define the symplectic form $\omega_{\cX,\epsilon}$ (degenerate when $\epsilon=0$) on $\cX$ by
\begin{equation*}
\omega_{\cX,\epsilon}:=(\omega_{\bC^4}+\epsilon^2\omega_{\bP^1})\mid_\cX.
\end{equation*}

Let $B(\epsilon)=\{(y,z)\in\bC^2\mid |y|^2+|z|^2\leq\epsilon^2\}\subset\bC^2$ be the ball of radius $\epsilon$. Consider the radial map $\rho_\epsilon:\bC^2\setminus\{0\}\to\bC^2\setminus B(\epsilon)$,
\begin{equation*}
\rho_\epsilon(y,z)=
\frac{\sqrt{|y|^2+|z|^2+\epsilon^2}}{\sqrt{|y|^2+|z|^2}}(y,z).
\end{equation*}
Let $\varrho_\epsilon
=\mathrm{id}_{\bC^2}\times\rho_\epsilon:\bC^2\times(\bC^2\setminus\{0\})\to
\bC^2\times(\bC^2\setminus B(\epsilon))$. Then $\varrho_\epsilon$ preserves the conifold $\cY_0$ and it maps $\cY_0\setminus\{0\}$ to $\cY_0(\epsilon):=\cY_0\setminus(\cY_0\cap(\bC^2\times B(\epsilon)))$. By the results in \cite{McSa} and \cite{DSV}, the map
\begin{equation*}
\psi_\epsilon:=\varrho_\epsilon\mid_{\cY_0\setminus\{0\}}\circ p\mid_{\cX\setminus\bP^1}: \cX\setminus\bP^1\to \cY_0(\epsilon)
\end{equation*}
is a symplectomorphism.

Define the path $\gamma^+_\epsilon$ by
\begin{equation*}
\gamma^+_\epsilon:=\phi_0\circ\varrho_\epsilon\circ\phi_0^{-1}\circ
\gamma^+:S^1\to T^*S^3.
\end{equation*}
By applying the construction in (\ref{lag}) to the path $\gamma^+_\epsilon$, we obtain a Lagrangian $M_{\gamma^+_\epsilon}$ in $T^*S^3$. Then we define the Lagrangian $L_\epsilon$ in the resolved conifold $\cX$ to be
\begin{equation*}
L_\epsilon:=\psi_\epsilon^{-1}(\phi_0^{-1}(M_{\gamma^+_\epsilon})).
\end{equation*}
The Lagrangian $L_\epsilon$ will be used to define the open Gromov-Witten invariants in Section \ref{sec:A-model}. Sometimes we omit the parameter $\epsilon$ and denote $L_\epsilon$ by $L_K$ for the torus knot $K$, or simply by $L_{P,Q}$. Let $\cX_\epsilon$ be $\cX$ equipped with the K\"ahler structure $\omega_{\cX,\epsilon}$ for $\epsilon\geq 0$, which is degenerate when $\epsilon=0$. We summarize the construction as the following diagram (\cite[Equation 2.17]{DSV}). Notice that the roles of $\cX,\cY$ and $L,M$ are reversed from \cite{DSV}. In the diagram, the wiggly arrow for $\cX_\epsilon$ means changing the K\"ahler structure $\omega_{\cX,\epsilon}$ while keeping the underlying complex manifold $\cX$. The wiggly arrow for $\cY_\delta$ means changing the deformation parameter $\delta$, which causes $\cY_\delta$ to collapse to $\cY_0$ when $\delta=0$. Here $p\vert_{L_0}$ is a diffeomorphism between $L_0$ and $M_0$.

\begin{equation*}
  \raisebox{-0.2\height}{\includegraphics{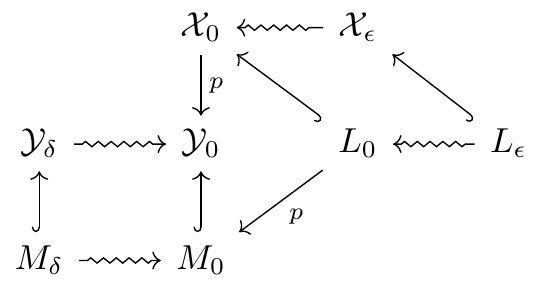}}
\end{equation*}

\section{Topological A-strings in the resolved conifold: Gromov-Witten theory}
\label{sec:A-model}

\subsection{Equivariant cohomology of the resolved conifold}
Let $\cX\cong[\cO_{\bP^1}(-1)\oplus\cO_{\bP^1}(-1)\to \bP^1]$ be the
resolved conifold given by the equation $xs=wt,\ ys=zt$ where $[s:t]$ are the homogeneous coordinates on $\bP^1$. Consider the 1-dimensional torus $\T\cong\bC^*$ which acts on $\cX$ in the following way. Let torus $\T$ act on $\bP^1$ by
\begin{equation*}
t\cdot[s:t]=[t^{-P-Q}s:t].
\end{equation*}
Then there are two $\T-$fixed points $\fp_0=[0:1],\fp_1=[1:0]$. Let $\iota_0:\fp_0\hookrightarrow \cX$ and $\iota_1:\fp_1\hookrightarrow \cX$ be the inclusion maps of $\fp_0$ and $\fp_1$ respectively.

We lift this action to an $\T$ action on $\cX$ by choosing the weights of the fibers of each $\cO(-1)$ at $\fp_0,\fp_1$ to be $P,-Q$ and $Q,-P$ respectively. Let
\begin{equation}
\Rtt:=H_{\T}^*(\pt;\bC)=\bC[\sv]
\end{equation}
be the $\T-$equivariant cohomology of a point. Then the $\T-$equivariant cohomology ring $H_{\T}^*(\cX;\bC)$ can be written as
\begin{equation}
H_{\T}^*(\cX;\bC)=\bC[\sv,\HT]/\langle \HT(\HT-(P+Q)\sv)\rangle.
\end{equation}
Here we have $\deg \HT=\deg \sv=2$ and $\HT\mid_{\fp_0}=0, \HT\mid_{\fp_1}=(P+Q)\sv$.

On the other hand, we define
\begin{eqnarray*}
\phi_0 :&=&\frac{[\fp_0]}{e_{T_{P,Q}}(T_{p_0}\cX)}
=\frac{(\HT-(P+Q)\sv)(\HT-P\sv)(\HT-Q\sv)}{-(P+Q)\sv P\sv Q\sv}=\frac{\HT-(P+Q)\sv}{-(P+Q)\sv}\\
\phi_1 :&=&\frac{[\fp_1]}{e_{T_{P,Q}}(T_{p_1}\cX)}
=\frac{\HT(\HT-P\sv)(\HT-Q\sv)}{(P+Q)\sv P\sv Q\sv}=\frac{\HT}{(P+Q)\sv},
\end{eqnarray*}
where $[\mathfrak p_\alpha]$ is the equivariant Poincar\'e dual of $\mathfrak p_\alpha$. Then $\{\phi_0, \phi_1\}$ is a basis of $H_{\T}^*(\cX;\bC)\otimes_{\bC[\sv]}\bC(\sv)$. We have
\begin{equation}
\phi_\alpha \cup \phi_\beta = \delta_{\alpha \beta}\phi_\alpha.
\end{equation}
Therefore $\{\phi_0, \phi_1\}$ is a canonical basis of $H_{\T}^*(\cX;\bC)\otimes_{\bC[\sv]}\bC(\sv)$. The $\T-$equivariant Poincar\`e pairing is given by
\begin{equation}
(\phi_\alpha, \phi_\beta)_{\T} = \frac{\delta_{\alpha \beta}}{\Delta^\alpha},\qquad\alpha,\beta\in\{0,1\},
\end{equation}
where $\Delta^\alpha=e_{T_{P,Q}}(T_{\mathfrak p_\alpha}\cX)=(-1)^{\alpha+1}(P+Q) P Q\sv^3,\alpha=0,1$.

The dual basis $\{\phi_\alpha\}$ are $\{\phi^\alpha=\Delta^\alpha\phi_\alpha =[\mathfrak p_\alpha]\}$ under the $\T-$equivariant Poincar\`e pairing. The normalized canonical basis $\{\hat\phi_0, \hat\phi_1\}$ is defined to be $\hat\phi_\alpha=\sqrt{\Delta^\alpha}\phi_\alpha$. Then we have
\begin{equation}
(\hat\phi_\alpha, \hat\phi_\beta)_{\T} = \delta_{\alpha \beta},\qquad\alpha,\beta\in\{0,1\}.
\end{equation}
Let $\bStt$ be a finite extension of the field $\bC(\sv)$ by including $\sqrt{\Delta^\alpha},\alpha=0,1$. Then $\{\hat\phi_0, \hat\phi_1\}$ is a basis of
\begin{equation}
H_{\T}^*(\cX;\bC)\otimes_{\bC[\sv]}\bStt.
\end{equation}

\subsection{Equivariant Gromov-Witten invariants and their generating functions}
Let $d\in H_2(\cX,\bZ)$ be an effective curve class. Let $\Mbar_{g,n}(\cX,d)$ be the moduli space of genus $g$, n-pointed, degree $d$ stable maps to $\cX$. Given $\gamma_1,\ldots, \gamma_n\in H_\T^*(\cX,\bC)$ and $a_1,\ldots,a_n\in \bZ_{\geq 0}$, we define genus $g$, degree $d$ $\T$-equivariant descendant Gromov-Witten invariants of $\cX$:
$$
\langle \tau_{a_1}(\gamma_1)\ldots\tau_{a_n}(\gamma_n)\rangle_{g,n,d}^{\cX,\T}:=
\int_{[\Mbar_{g,n}(\cX,d)^\T]^\vir} \frac{\prod_{j=1}^n \psi_j^{a_j}\ev_j^*(\gamma_j)\mid_{\Mbar_{g,n}(\cX,d)^\T}}{e_\T(N^\vir)}
$$
where $\ev_j: \Mbar_{g,n}(X,d) \to \cX$ is the evaluation at the $j$-th marked point,
which is a $\T$-equivariant map, $\Mbar_{g,n}(\cX,d)^\T$ is the $\T$-fixed locus, and $e_\T(N^\vir)$ is the $\T$-equivariant Euler class of the virtual normal bundle. We also define genus $g$, degree $d$ primary Gromov-Witten invariants:
$$
\langle \gamma_1,\ldots, \gamma_n\rangle_{g,n,d}^{\cX,\T}:=
\langle \tau_0(\gamma_1) \cdots \tau_0(\gamma_n)\rangle_{g,n,d}^{\cX,\T}.
$$

Define the Novikov ring
\[
\nov:=\widehat{\bC[E(\cX)]}=\{\sum_{d\in E(\cX)} c_d \fQ^d: c_d\in \bC\},
\]
where $E(\cX)$ is the set of effective curve classes which is identified with the set of nonnegative integers in our case. We also define the following double correlator with primary insertions:
\begin{eqnarray*}
\llangle \gamma_1\psi^{a_1},  \cdots, \gamma_n\psi^{a_n} \rrangle^{\cX,\T}_{g,n}
:=\sum_{m=0}^\infty \sum_{d\in E(\cX)}\frac{\fQ^d}{m!}\langle
\gamma_1\psi^{a_1}, \cdots, \gamma_n\psi^{a_n}, t^m \rangle^{\cX,\T}_{g,n+m,d}
\end{eqnarray*}
where $\fQ^d\in \bC[E(\cX)]\subset \nov$ is the Novikov variable corresponding to
$d\in E(\cX)$, and $t\in H^*_{\T}(\cX;\bC)\otimes_\Rtt \bStt$. Let $t=t^0 1 + t^1 \HT=\that^0\phihat_0+\that^1\phihat_1$. As a convention, we use $\btau\in H^*_\T(\cX;\bC)$ to denote a class in degree $2$. Then $\btau=\tau_0 1 +\tau_1 \H^{T_{P,Q}}$ where $\tau_1\in \bC$, and $\tau_0$ is degree $2$ in $H^*_{\T}(\pt)$.

For $j=1,\ldots,n$, introduce formal variables
$$
\bu_j =\bu_j(z)= \sum_{a\geq 0}(u_j)_a z^a
$$
where $(u_j)_a \in H^*_{\T}(\cX;\bC)\otimes_\Rtt\bStt$.
Define
\begin{align*}
\llangle \bu_1,\ldots, \bu_n \rrangle_{g,n}^{\cX,\T} &=
\llangle \bu_1(\psi),\ldots, \bu_n(\psi) \rrangle_{g,n}^{\cX,\T}\\
&=\sum_{a_1,\ldots,a_n\geq 0}
\llangle (u_1)_{a_1}\psi^{a_1}, \cdots, (u_n)_{a_n}\psi^{a_n}\rrangle_{g,n}^{\cX,\T}.
\end{align*}

\subsection{The equivariant quantum cohomology of $X$}\label{sec:QH}
In order to define the equivariant quantum cohomology of $\cX$, we consider the three-point double correlator $\llangle a,b,c\rrangle_{0,3}^{\cX,\T}$ for $a,b,c \in H^*_{\T}(\cX)\otimes_\Rtt \bStt$. Then by divisor equation, we have
\begin{equation}
\llangle a,b,c\rrangle_{0,3}^{\cX,\T}\in \bStt \llbracket \tfQ \rrbracket,
\end{equation}
where $\tfQ=\fQ e^{t^1}$. We define quantum product $\star_t$ by
\begin{equation}
(a\star_t b,c)_{\cX,\T}:=\llangle a,b,c\rrangle_{0,3}^{\cX,\T}.
\end{equation}

Set
$$
\novtt:= \bStt\otimes_{\bC}\nov =\bStt \llbracket E(\cX)\rrbracket.
$$
Then $H:=H^*_{\T}(\cX;\novtt)$ is a free $\novtt$-module of rank $2$.
Any point $t\in H$ can be written as  $t=\sum_{\alpha=0}^{1}\that^{\alpha} \hat{\phi}_{i}$
. We have
$$
H=\mathrm{Spec}(\novtt [ \that^{0},\that^{1}]).
$$
Let $\hat{H}$ be the formal completion of $H$ along the origin:
$$
\hat{H} :=\mathrm{Spec}(\novtt\llbracket \that^{0},\that^{1}\rrbracket).
$$
Let $\cO_{\hat{H}}$ be the structure sheaf on $\hat{H}$, and let $\cT_{\hat{H}}$ be the tangent sheaf on
$\hat{H}$.
Then $\cT_{\hat{H}}$ is a sheaf of free $\cO_{\hat{H}}$-modules of rank $2$. Given an open set in $\hat{H}$,
$$
\cT_{\hat{H}}(U)  \cong \bigoplus_{\alpha}\cO_{\hat{H}}(U) \frac{\partial}{\partial \that^{\alpha}}.
$$
The quantum product and the $\T$-equivariant Poincar\'{e} pairing defines the structure of a formal
Frobenius manifold on $\hat{H}$:
$$
\frac{\partial}{\partial \that^{\alpha}} \star_t \frac{\partial}{\partial \that^{\alpha'}}
=\sum_{\beta} \llangle \hat{\phi}_{\alpha},\hat{\phi}_{\alpha'},\hat{\phi}_{\beta}\rrangle_{0,3}^{\cX,\T}
\frac{\partial}{\partial \that^{\beta}}
\in \Gamma(\hat{H}, \cT_{\hat{H}}).
$$
$$
( \frac{\partial}{\partial \that^{\alpha}},\frac{\partial}{\partial \that^{\alpha,}})_{\cX,\T} =\delta_{\alpha,\alpha'}.
$$

The set of global sections $\Gamma(\hat{H},\cT\hat{H})$ is a free $\cO_{\hat{H}}(\hat{H})$-module of rank $2$:
$$
\Gamma(\hat{H},\cT\hat{H})=
\bigoplus_{\alpha=0,1}\cO_{\hat{H}}(\hat{H})\frac{\partial}{\partial \that^{\alpha}}.
$$
Under the quantum product $\star_t$, the triple $(\Gamma(\hat{H},\cT\hat{H}),\star_t,(, )_{\cX,\T})$ is a Frobenius algebra over the ring $\cO_{\hat{H}}(\hat{H})=\novtt\llbracket \that^{0},\that^{1}\rrbracket$. The triple $(\Gamma(\hat{H},\cT\hat{H}),\star_t,(, )_{\cX,\T})$ is called the \emph{quantum cohomology} of $\cX$ and is denoted by $QH^*_{\T}(\cX)$.

The semi-simplicity of the classical cohomology $H^*_{\T}(\cX;\bC)\otimes_\Rtt \bStt$ implies the semi-simplicity of the quantum cohomology $QH^*_{\T}(\cX)$. In fact, there exists a canonical basis $\{\phi_0(t),\phi_1(t)\}$ of $QH^*_{\T}(\cX)$ characterized by the property that
\begin{equation}
\phi_\alpha(t)\to  \phi_\alpha,\quad\mathrm{when }\quad t,\fQ\to 0,\quad \alpha=0,1.
\end{equation}
We denote $\{\phi^\alpha(t)\}$ to be the dual basis to $\{\phi_\alpha(t)\}$ with respect to the metric $(,)_{\cX,\T}$. See \cite{LP} for more general discussions on the canonical basis.

\subsection{The A-model canonical coordinates and the $\Psi$-matrix}\label{sec:A-canonical}
The canonical coordinates $\{ u^{\alpha}=u^{\alpha}(t)|\alpha=0,1\}$ on the formal Frobenius
manifold $\hat{H}$ are characterized by
\begin{equation}\label{eqn:partial-u}
\frac{\partial}{\partial u^{\alpha}} = \phi_{\alpha}(t).
\end{equation}
up to additive constants in $\novtt$. We choose canonical coordinates
such that they lie in $\bStt[t^1]\llbracket {\tfQ},t^0\rrbracket$ and vanish when $\fQ=0$,
$t^0=t^1=0$. Then $u^{\alpha}-\sqrt{\Delta^{\alpha}}t^{\alpha} \in \bStt[t^1]\llbracket \tfQ,t^0\rrbracket$
and vanish when $\tfQ=0$, $t^0=0$.

We define $\Delta^{\alpha}(t)\in \bStt\llbracket \tfQ,t^0\rrbracket$ by the following equation:
$$
(\phi_{\alpha}(t), \phi_{\alpha'}(t))_{\cX,\T} =\frac{\delta_{\alpha,\alpha'}}{\Delta^{\alpha}(t)}.
$$
Then $\Delta^{\alpha}(t) \to \Delta^{\alpha}$  in the large radius limit
$\tfQ,t^0\to 0$. The normalized canonical basis of $(\hat{H},\star_t)$ is
$$
\{ \hat{\phi}_{\alpha}(t):= \sqrt{\Delta^{\alpha}(t)}\phi_{\alpha}(t)| \alpha=0,1\}.
$$
They satisfy
$$
\hat{\phi}_{\alpha}(t)\star_t \hat{\phi}_{\alpha'}(t) =\delta_{\alpha, \alpha'}\sqrt{\Delta^{\alpha}(t)}\hat{\phi}_{\alpha}(t),\quad
(\hat{\phi}_{\alpha}(t), \hat{\phi}_{\alpha'}(t))_{\cX,\T}=\delta_{\alpha,\alpha'}.
$$
(Note that $\sqrt{\Delta^{\alpha}(t)}= \sqrt{\Delta^{\alpha}} \cdot
\sqrt{\frac{ \Delta^{\alpha}(t) }{\Delta^{\alpha}} }$,
where
$\sqrt{\Delta^{\alpha}}\in \bStt$ and
$\sqrt{\frac{ \Delta^{\alpha}(t) }{\Delta^{\alpha}} }\in \bStt\llbracket\tfQ,t^0\rrbracket$,
so $\sqrt{\Delta^{\alpha}(t)}\in \bStt\llbracket\tfQ,t^0\rrbracket$.)
We call $\{\hat{\phi}_{\alpha}(t)|\alpha=0,1\}$ the {\em quantum} normalized canonical basis
to distinguish it from the  {\em classical} normalized canonical basis $\{ \hat{\phi}_{\alpha}|\alpha =0,1\}$.
The quantum canonical basis tends to the classical canonical
basis in the large radius limit: $\hat{\phi}_{\alpha}(t)\to \hat{\phi}_{\alpha}$ as $\tfQ,t^0\to 0$.

Let $\Psi=(\Psi_{\alpha'}^{\  \alpha})$ be the transition matrix between the classical and quantum
normalized canonical bases:
\begin{equation}\label{eqn:Psi-phi}
\hat{\phi}_{\alpha'}=\sum_{\alpha=0,1} \Psi_{\alpha'}^{\ \alpha} \hat{\phi}_\alpha(t).
\end{equation}
Then $\Psi$ is an $2\times 2$ matrix with entries
in $\bStt\llbracket\tfQ,t^0\rrbracket$, and $\Psi\to \one$ (the identity matrix)
in the large radius limit $\tfQ,t^0\to 0$. Both
the classical and quantum normalized canonical bases are orthonormal with respect
to the $\T$-equivariant Poincar\'{e} pairing $(\ ,\ )_{\cX,\T}$, so $\Psi^T\Psi= \Psi \Psi^T= \one$,
where $\Psi^T$ is the transpose of $\Psi$, or equivalently
$$
\sum_{\beta=0,1} \Psi_{\beta}^{\ \alpha} \Psi_{\beta}^{\ \alpha'} =\delta_{\alpha,\alpha'}
$$
Equation \eqref{eqn:Psi-phi} can be rewritten as
$$
\frac{\partial}{\partial \that^{\alpha'}} =\sum_{\alpha=0,1} \Psi_{\alpha'}^{\ \alpha}
\sqrt{\Delta^{\alpha}(t)} \frac{\partial}{\partial u^{\alpha}}
$$
which is equivalent to
\begin{equation}
\label{eqn:Psi-matrix}
\frac{du^{\alpha}}{\sqrt{\Delta^{\alpha}(t)}} =
\sum_{\alpha'=0,1}
d\that^{\alpha'} \Psi_{\alpha'}^{\ \alpha}.
\end{equation}

\subsection{The equivariant quantum differential equation}
We consider the Dubrovin connection $\nabla^z$, which is a family
of connections parametrized by $z\in \bC\cup \{\infty\}$, on the tangent bundle
$T_{\hat{H}}$ of the formal Frobenius manifold $\hat{H}$:
$$
\nabla^z_{\alpha}=\frac{\partial}{\partial \hat t^{\alpha}} -\frac{1}{z} \hat{\phi}_{\alpha}\star_t
$$
The commutativity (resp. associativity)
of $*_t$ implies that $\nabla^z$ is a torsion
free (resp. flat) connection on $T_{\hat{H}}$ for all $z$. The equation
\begin{equation}\label{eqn:qde}
\nabla^z \mu=0
\end{equation}
for a section $\mu\in \Gamma(\hat{H},\cT_{\hat{H}})$ is called the {\em $\T$-equivariant
quantum differential equation} ($\T$-equivariant QDE). Let
$$
\cT_{\hat{H}}^{f,z}\subset \cT_{\hat{H}}
$$
be the subsheaf of flat sections with respect to the connection $\nabla^z$.
For each $z$, $\cT_{\hat{H}}^{f,z}$ is a sheaf of
$\novtt$-modules of rank $2$.

A section $L\in \End(T_{\hat{H}})=\Gamma(\hat{H},\cT_{\hat{H}}^*\otimes\cT_{\hat{H}})$
defines a $\cO_{\hat{H}}(\hat{H})$-linear map
$$
L: \Gamma(\hat{H},\cT_{\hat{H}})= \bigoplus_{\alpha} \cO_{\hat{H}}(\hat{H})
\frac{\partial}{\partial \that^{\alpha}}
\to \Gamma(\hat{H},\cT_{\hat{H}})
$$
from the free $\cO_{\hat{H}}(\hat{H})$-module $\Gamma(\hat{H},\cT_{\hat{H}})$ to itself.
Let $L(z)\in \End(T_{\hat H})$ be a family of endomorphisms of the tangent bundle $T_{\hat{H}}$
parametrized by $z$. $L(z)$ is called a {\em fundamental solution} to the $\T$-equivariant QDE if
the $\cO_{\hat{H}}(\hat{H})$-linear map
$$
L(z): \Gamma(\hat{H},\cT_{\hat{H}}) \to \Gamma(\hat{H},\cT_{\hat{H}})
$$
restricts to a $\novtt$-linear isomorphism
$$
L(z): \Gamma(\hat{H},\cT_H^{f,\infty})=\bigoplus_{\alpha} \novtt \frac{\partial}{\partial \that^{\alpha}}
\to \Gamma(\hat{H},\cT_H^{f,z}).
$$
between rank $2$ free $\novtt$-modules.

\subsection{The $\cS$-operator}\label{sec:A-S}
The $\cS$-operator is defined as follows.
For any cohomology classes $a,b\in H_{\T}^*(\cX;\bStt)$,
$$
(a,\cS(b))_{\cX,\T}=(a,b)_{\cX,\T}
+\llangle a,\frac{b}{z-\psi}\rrangle^{\cX,\T}_{0,2}
$$
where
$$
\frac{b}{z-\psi}=\sum_{i=0}^\infty b\psi^i z^{-i-1}.
$$
The $\cS$-operator can be viewed as an element in $\End(T_{\hat{H}})$ and is a fundamental solution to the $\T$-equivariant
big QDE \eqref{eqn:qde}.  The proof for $\cS$ being a fundamental solution can be found in \cite{CK}.

\begin{remark}
One may notice that since there is a formal variable $z$ in the definition of
the $\T$-equivariant big QDE \eqref{eqn:qde}, one can consider its solution space over different rings. Here the operator
$\cS= \one+ \cS_1/z+ \cS_2/z^2+\cdots$ is viewed as a formal power series in $1/z$ with operator-valued coefficients.
\end{remark}

\begin{remark}\label{Novikov}
By divisor equation and string equation, we have
$$
(a,b)_{\cX,\T}+\llangle a,\frac{b}{z-\psi}\rrangle^{\cX,\T}_{0,2}=(a,be^{(t^0 1+t^1H)/z})_{\cX,\T}+
 \sum_{d>0}\fQ^de^{dt^1}\langle a,\frac{be^{(t^0 1+t^1H)/z}}{z-\psi}\rangle^{\cX,\T}_{0,2,d}.
$$
In the above expression, if we fix the power of $z^{-1}$, then only finitely many terms in the expansion of $e^{(t^0 1+t^1H)/z}$ contribute. Therefore, the factor $e^{dt^1}$ can play the role of $\fQ^d$ and hence the restriction $\llangle a,\frac{b}{z-\psi}\rrangle^{\cX,\T}_{0,2}|_{\fQ=1}$ is well-defined. So the operator $\cS|_{\fQ=1}$ is well-defined.
\end{remark}

\begin{definition}[$\T$-equivariant $J$-function] \label{big-J}
The $\T$-equivariant big $J$-function $J_{\T}(z)$  is characterized by
$$
(J_{\T}(z),a)_{\cX,\T} = (1,\cS(a))_{\cX,\T}
$$
for any $a\in H_{\T}^*(X;\bStt)$. Equivalently,
$$
J_{\T}(z)= 1+ \sum_{\alpha}\llangle 1, \frac{\hat{\phi}_{\alpha}}{z-\hat{\psi}}\rrangle_{0,2}^{\cX,\T}\hat{\phi}_{\alpha}.
$$
\end{definition}

We consider several different (flat) basis for $H_{\T}^*(\cX;\bStt)$:
\begin{enumerate}
\item The classical canonical basis $\{ \phi_{\alpha}|\alpha=0,1 \}$.
\item The basis dual to the classical canonical basis with respect to the $\T$-equivariant Poincar\`e pairing:
$\{ \phi^{\alpha} =\Delta^{\alpha} \phi_{\alpha}|\alpha=0,1 \} $.
\item The classical normalized canonical basis
$\{ \hat{\phi}_{\alpha}=\sqrt{\Delta^{\alpha}}\phi_{\alpha} |\alpha=0,1\}$ which is self-dual: $\{ \hat{\phi}^{\alpha}=\hat{\phi}_{\alpha}|\alpha=0,1 \}$.
\end{enumerate}

For $\alpha, \alpha'\in \{0,1\}$, define
$$
S^{\alpha'}_{\ \alpha}(z) := (\phi^{\alpha}, \cS(\phi_{\alpha})).
$$
Then $(S^{\alpha'}_{\ \alpha}(z))$ is the matrix  of the $\cS$-operator with respect to the canonical basis
$\{\phi_{\alpha}|\alpha=0,1 \}$:
\begin{equation}\label{eqn:S}
\cS(\phi_{\alpha}) =\sum_{\alpha'=0,1}
\phi_{\alpha'} S^{\alpha'}_{\ \alpha}(z).
\end{equation}

For $\alpha,\alpha'\in \{0,1\}$, define
$$
S_{\alpha'}^{\ \widehat{\alpha} }(z) := (\phi_{\alpha'}, \cS(\hat{\phi}^{\alpha})).
$$
Then $(S_{\alpha'}^{\  \widehat{\alpha}})$ is the matrix of the $\cS$-operator
with respect to the bases $\{\hat{\phi}^{\alpha}|\alpha=0,1\}$ and
$\{\phi^{\alpha}|\alpha=0,1\}$:
\begin{equation}\label{eqn:barS}
\cS(\hat{\phi}^{\alpha})=\sum_{\alpha'=0,1} \phi^{\alpha'}
 S_{\alpha'}^{\ \widehat{\alpha}}(z).
\end{equation}

Introduce
\begin{align*}
S_z(a,b)&=(a,\cS(b))_{\cX,\T},\\
V_{z_1,z_2}(a,b)&=\frac{(a,b)_{\cX,\T}}{z_1+z_2}+\llangle \frac{a}{z_1-\psi_1},
                  \frac{b}{z_2-\psi_2}\rrangle^{\cX,\T}_{0,2}.
\end{align*}
A well-known WDVV-like argument says
\begin{equation}
\label{eqn:two-in-one}
V_{z_1,z_2}(a,b)=\frac{1}{z_1+z_2}\sum_i S_{z_1}(T_i,a)S_{z_2}(T^i,b),
\end{equation}
where $T_i$ is any basis of $H^*_{\T}(\cX;\bStt)$ and $T^i$ is its dual basis.
In particular,
$$
V_{z_1,z_2}(a,b)=\frac{1}{z_1+z_2}\sum_{\alpha=0,1} S_{z_1}(\hat{\phi}_{\alpha},a)S_{z_2}(\hat{\phi}_{\alpha},b).
$$

\subsection{The A-model $R$-matrix}
Let $U$ denote the diagonal matrix whose diagonal entries are the canonical coordinates.
The results in \cite{Gi01a} imply the following statement.

\begin{theorem}\label{R-matrix}
There exists a unique matrix power series $R(z)= \one + R_1z+R_2 z^2+\cdots$
satisfying the following properties.
\begin{enumerate}
\item The entries of $R_d$ lie in $\bStt\llbracket\tfQ,t^0\rrbracket$.
\item $\tS=\Psi R(z) e^{U/z}$  is a fundamental solution to the $\T$-equivariant
QDE \eqref{eqn:qde}.
\item $R$ satisfies the unitary condition $R^T(-z)R(z)=\one$.
\item
\begin{equation}\label{eqn:R-at-zero}
\begin{aligned}
\lim_{\tfQ,t^0\to 0} R_\alpha^{\,\ \beta}(z) =\delta_{\alpha\beta} \prod_{i=1}^3\exp\big(-\sum_{n=1}^\infty \frac{B_{2n}}{2n(2n-1)}
(\frac{z}{\sw_i(\alpha)})^{2n-1} \big),
\end{aligned}
\end{equation}
where $\sw_1(0)=-(P+Q)\sv,\sw_2(0)=P\sv,\sw_3(0)=Q\sv,\sw_1(1)=(P+Q)\sv,
\sw_2(1)=-P\sv,\sw_3(1)=-Q\sv.$
\end{enumerate}
\end{theorem}

Each matrix in (2) of Theorem \ref{R-matrix} represents an operator with respect to the classical canonical basis  $\{ \hat{\phi}_{\alpha}|\alpha=0,1\}$.
So $R^T$ is the adjoint of $R$ with respect to the $\T$-equivariant
Poincar\'{e} pairing $(\ , \ )_{\cX,\T}$.

We call the unique $R(z)$ in Theorem \ref{R-matrix} the {\em A-model $R$-matrix}.
The A-model $R$-matrix plays a central role in the quantization formula of the descendant potential of $\T$-equivariant Gromov-Witten
theory of $\cX$. We will state this formula in terms of graph sum in the the next subsection.

\subsection{The A-model graph sum} \label{sec:Agraph}
For $\alpha=0,1$, let
$$
\hat{\phi}_\alpha(t):= \sqrt{\Delta^\alpha(t)}\phi_\alpha(t).
$$
Then $\hat{\phi}_1(t)$, $\hat{\phi}_2(t)$ is the normalized canonical basis
of $QH_\T^*(\cX)$, the $\T$-equivariant quantum cohomology of $\cX$. Define
$$
{S}^{\hualpha}_{\, \ \hubeta}(z) := (\hat{\phi}_\alpha(t), \cS(\hat{\phi}_\beta(t))).
$$
Then $({S}^{\underline{\hat \alpha}}_{\, \ \underline{\hat{\beta}}}(z))$ is the matrix of the $\cS$-operator with
respect to the ordered basis $(\hat{\phi}_1(t),\hat{\phi}_2(t))$:
\begin{equation}\label{eqn:mathringS}
\cS(\hat{\phi}_\beta(t))=\sum_{\alpha=0}^1 \hat{\phi}_\alpha(t) {S}^\hualpha_{\,\ \hubeta}(z).
\end{equation}

Given a connected graph $\Ga$, we introduce the following notation.
\begin{enumerate}
\item $V(\Ga)$ is the set of vertices in $\Ga$.
\item $E(\Ga)$ is the set of edges in $\Ga$.
\item $H(\Ga)$ is the set of half edges in $\Gamma$.
\item $L^o(\Ga)$ is the set of ordinary leaves in $\Ga$.
\item $L^1(\Ga)$ is the set of dilaton leaves in $\Ga$.
\end{enumerate}

With the above notation, we introduce the following labels:
\begin{enumerate}
\item (genus) $g: V(\Ga)\to \bZ_{\geq 0}$.
\item (marking) $\beta: V(\Ga) \to \{0,1\}$. This induces
$\beta:L(\Ga)=L^o(\Ga)\cup L^1(\Ga)\to \{0,1\}$, as follows:
if $l\in L(\Ga)$ is a leaf attached to a vertex $v\in V(\Ga)$,
define $\beta(l)=\beta(v)$.
\item (height) $k: H(\Ga)\to \bZ_{\geq 0}$.
\end{enumerate}

Given an edge $e$, let $h_1(e),h_2(e)$ be the two half edges associated to $e$. The order of the two half edges does not affect the graph sum formula in this paper. Given a vertex $v\in V(\Ga)$, let $H(v)$ denote the set of half edges
emanating from $v$. The valency of the vertex $v$ is equal to the cardinality of the set $H(v)$: $\val(v)=|H(v)|$.
A labeled graph $\vGa=(\Ga,g,\beta,k)$ is {\em stable} if
$$
2g(v)-2 + \val(v) >0
$$
for all $v\in V(\Ga)$.

Let $\bGa(X)$ denote the set of all stable labeled graphs
$\vGa=(\Gamma,g,\beta,k)$. The genus of a stable labeled graph
$\vGa$ is defined to be
$$
g(\vGa):= \sum_{v\in V(\Ga)}g(v)  + |E(\Ga)|- |V(\Ga)|  +1
=\sum_{v\in V(\Ga)} (g(v)-1) + (\sum_{e\in E(\Gamma)} 1) +1.
$$
Define
$$
\bGa_{g,n}(\cX)=\{ \vGa=(\Gamma,g,\beta,k)\in \bGa(X): g(\vGa)=g, |L^o(\Ga)|=n\}.
$$
Given $\alpha\in \{0,1\}$ and $j\in\{1,\cdots,n\}$, define
$$
\bu_j^\alpha(z) = \sum_{a\geq 0} (u_j)^\alpha_a z^a.
$$

We assign weights to leaves, edges, and vertices of a labeled graph $\vGa\in \bGa(\cX)$ as follows.

\begin{enumerate}
\item {\em Ordinary leaves.} To each ordinary leaf $l_j \in L^o(\Ga)$ with  $\beta(l_j)= \beta\in \{0,1\}$
and  $k(l_j)= k\in \bZ_{\geq 0}$, we assign:
$$(\cL^{\bu})^\beta_k(l_j) = [z^k] (\sum_{\alpha,\gamma=0,1} \frac{\bu^\alpha_j(z)}{\sqrt{\Delta^\alpha(t)}}
{S}^{\hugamma}_{\,\ \hualpha}(z) R(-z)_\gamma^{\,\ \beta} ).
$$
\item {\em Dilaton leaves.} To each dilaton leaf $l \in L^1(\Ga)$ with $\beta(l)=\beta \in \{0,1\}$
and $2\leq k(l)=k \in \bZ_{\geq 0}$, we assign
$$
(\cL^1)^\beta_k(l) = [z^{k-1}](-\sum_{\alpha=0,1}\frac{1}{\sqrt{\Delta^\alpha(t)}} R_\alpha^{\,\ \beta}(-z)).
$$

\item {\em Edges.} To an edge connected a vertex marked by $\alpha\in \{0,1\}$ to a vertex
marked by $\beta\in \{0,1\}$ and with heights $k$ and $l$ at the corresponding half-edges, we assign
$$
\cE^{\alpha,\beta}_{k,l}(e) = [z^k w^l]
\Bigl(\frac{1}{z+w} (\delta_{\alpha,\beta}-\sum_{\gamma=0,1}
R_\gamma^{\,\ \alpha}(-z) R_\gamma^{\,\ \beta}(-w)\Bigr).
$$
\item {\em Vertices.} To a vertex $v$ with genus $g(v)=g\in \bZ_{\geq 0}$ and with
marking $\beta(v)=\beta$, with $n$ ordinary
leaves and half-edges attached to it with heights $k_1, ..., k_n \in \bZ_{\geq 0}$ and $m$ more
dilaton leaves with heights $k_{n+1}, \ldots, k_{n+m}\in \bZ_{\geq 0}$, we assign
$$
\big(\sqrt{\Delta^\beta(t)}\big)^{2g-2+n+m} \int_{\Mbar_{g,n+m}}\psi_1^{k_1} \cdots \psi_{n+m}^{k_{n+m}}.
$$
\end{enumerate}
Define the weight
\begin{eqnarray*}
w_A^{\bu}(\vGa) &=& \prod_{v\in V(\Ga)} (\sqrt{\Delta^{\beta(v)}(t)})^{2g(v)-2+\val(v)} \langle \prod_{h\in H(v)} \tau_{k(h)}\rangle_{g(v)}
\prod_{e\in E(\Ga)} \cE^{\beta(v_1(e)),\beta(v_2(e))}_{k(h_1(e)),k(h_2(e))}(e)\\
&& \cdot \prod_{j=1}^{n}(\cL^{\bu})^{\beta(l_j)}_{k(l_j)}(l_j)
\prod_{l\in L^1(\Ga)}(\cL^1)^{\beta(l)}_{k(l)}(l).
\end{eqnarray*}

With the above definition of the weight of a labeled graph, we have
the following theorem which expresses the $\T-$equivariant descendant
Gromov-Witten potential of $\cX$ in terms of graph sum.

\begin{theorem}[{Givental \cite{Gi01a}}]
\label{thm:Givental}
 Suppose that $2g-2+n>0$. Then
\begin{equation}
\label{eqn:A-graph-sum}
\llangle \bu_1,\ldots, \bu_n\rrangle_{g,n}^{\cX,\T}=\sum_{\vGa\in \bGa_{g,n}(\cX)}\frac{w_A^{\bu}(\vGa)}{|\Aut(\vGa)|}.
\end{equation}
\end{theorem}

\begin{remark}\label{descendant Q}In the above graph sum formula, we know that the restriction $S^{\widehat{\underline{\gamma}} }_{\
  \widehat{\underline{\alpha}}}(z)|_{\fQ=1}$ is well-defined by Remark
\ref{Novikov}. Meanwhile by (1) in Theorem \ref{R-matrix}, we know
that the restriction $R(z)|_{\fQ=1}$ is also well-defined. Therefore by Theorem \ref{thm:Givental}, $\llangle \bu_1,\ldots, \bu_n\rrangle_{g,n}^{\cX,\T}|_{\fQ=1}$ is well-defined.
\end{remark}

\subsection{Open-closed Gromov-Witten invariants}\label{sec:open-closed-GW}

Let $\cX=\cO_{\bP^1}(-1)\oplus \cO_{\bP^1}(-1)$ and $L_{P,Q}$ as defined in Section \ref{brane}. Given any partition of $\vec{\mu}$ with $l(\vec{\mu})=h$ and $g,d\in \bZ_{\geq 0}$, we define (cf. \cite[Section 4]{KaLiu})
\[
\overline{\cM}_{g,n,d,\vec{\mu}}:=\overline{\cM}_{g,n;h}(\cX,L_{P,Q}|d;\mu_1,\dots,\mu_{h}).
\]
The right hand side is the moduli space of stable maps
$u:(\Si,x_1,\cdots,x_n,\partial \Si)\to (\cX,L_{P,Q})$, where $\Si$ is a prestable
bordered Riemann surface of genus $g$ and $h$ boundary components
($\partial \Si=\sqcup_{j=1}^h R_j$ for each $R_j\cong S^1$) and $x_1,\cdots,x_n$ are interior marked points on $\Si$. We
require that $u_*[\Si]=d+(\sum_{i=1}^h \mu_i)\gamma_0$ and
$u_*[R_j]=\mu_j\gamma_0$. Here $\gamma_0$ is the generator of $H_1(\L;\bZ)$.

As shown in \cite{DSV}, the $(\T)_\bR$-action preserves the pair $(\cX,\L)$, and induces an action on $\overline{\cM}_{g,n,d,\vec{\mu}}$. The fixed locus consists of the map
\[
f: D_1\cup \dots D_h \cup C \to (\cX,\L),
\]
where each $D_i$ is a disk $\{t: |t|\leq 1\}$ and $C$ is a (possibly empty) nodal curve. The point $t=0$ on each $D_i$ is the only possible nodal point on $D_i$. The map $f|_{D_i}(t)=(t^{\mu_i Q},t^{\mu_i P},0)$ and $f|_C$ is a $\T$-invariant morphism. Here we use the local affine coordinates $(x,y,s/t)$ for $X$.

The $(\T)_\bR\cong S^1$-fixed part $\overline{\cM}_{g,n,d,\vec{\mu}}^{(\T)_\bR}$ is compact, and we define
\[
\langle \gamma_1,\dots,\gamma_n\rangle^{\cX,L_{P,Q},\T}_{g,n,d,\vec{\mu}}
=\int_{[\Mbar_{g,n,d,\vec{\mu}}^{(\T)_\bR}]^\vir}\frac{\prod_{j=1}^n \ev_j^*\gamma_j}{e_{(\T)_\bR}(N^\vir)}
\]
where $N^\vir$ is the virtual normal bundle of $\Mbar_{g,n,d,\vec{\mu}}^{(\T)_\bR}\subset \Mbar_{g,n,d,\vec{\mu}}$.

We define the disk factor for any $\mu\in \bZ_{>0}$ as
\begin{equation}
\label{eqn:disk-factor}
D(\mu)=(-1)^{(P+Q)\mu-1}\frac{\prod_{m=1}^{Q\mu-1}(P\mu+m)}{(Q\mu)!}.
\end{equation}
The localization computation in \cite{DSV} gives the following formula.
\begin{proposition}
\label{prop:localization}
\begin{equation}
\label{eqn:localization}
\langle \gamma_1,\ldots, \gamma_n\rangle_{g,n,d,\vmu}^{\cX, \L}
=\prod_{j=1}^h D(\mu_j) \quad \cdot
\int_{[\Mbar_{g,n+h}(\cX,d)^\T]^{\vir}} \frac{\prod_{i=1}^n\ev_i^*\gamma_i \prod_{j=1}^h\ev_{n+j}^*\phi^{0}}
{\prod_{j=1}^h(\frac{\sv}{\mu_j}
  -{\psi}_{n+j})}.
\end{equation}
\end{proposition}

For $a\in \bZ$, We introduce
\begin{align*}
\Phi_a(X)&=\sum_{\mu>0} D(\mu) (\frac{\mu}{\sv})^{a}\mu X^\mu,\quad \Phi^a(X)=\Phi_a(X) e_{T_{P,Q}}(T_{p_0}X),\\
\txi_0(z,X)&=\sum_{a\geq -2} z^a \Phi_a(X),\quad \txi^0(z,X)=\sum_{a\geq -2} z^a \Phi^a(X),\quad \txi^1(z,X)=\txi_1(z,X)=0.
\end{align*}

Let $\btau=\tau_0 1+ \tau_1 \HT\in H^*_{T_{P,Q}}(\cX;\bC)$ be a degree $2$ class. Define the A-model open potential associated to $(\cX,\L)$ as the following.
\[
F_{g,n}^{\cX,\L}(\btau; X_1,\dots,X_n)=\sum_{d\geq 0} \sum_{\mu_1,\dots,\mu_n>0} \sum_{\ell\geq 0} \frac{\langle\btau^\ell \rangle^{\cX,\L}_{g,\ell,d,\vmu}}{\ell!} X_1^{\mu_1}\dots X_n^{\mu_n}.
\]
The localization computation expresses these potentials as descendant potentials \cite{FLZ16}. When $2g-2+n>0$,
\begin{align}
\label{eqn:free-energy-localization-stable}
F_{g,n}^{\cX,\L}(\btau; X_1,\dots,X_n)&=\sum_{a_1,\dots,a_n\in\bZ_\geq 0} \sum_{\ell\geq 0} \sum_{d\geq 0} \frac{\langle \btau^\ell, \tau_{a_1}(\phi^0),\dots, \tau_{a_n}(\phi^0)\rangle^{\cX,T_{P,Q}}_{g,\ell+n,d} }{\ell!} \prod_{j=1}^n \Phi_{a_j}(X_j)\\
&=[z_1^{-1}\dots z_n^{-1}] \llangle \frac{\phi_0}{z_1-\psi_1}\dots \frac{\phi_0}{z_n-\psi_n}\rrangle_{g,n}^{\cX,T_{P,Q}}\vert_{\fQ=1} \prod_{j=1}^n \txi^0(z_j,X_j).\nonumber
\end{align}
The two special cases are
\begin{align}
F_{0,1}^{\cX,\L}(\btau;X)&=[z^{-2}]S_z(1,\phi_0)\txi^0(z,X),\label{eqn:free-energy-localization-disk}\\
F_{0,2}^{\cX,\L}(\btau;X_1,X_2)&=[z_1^{-1}z_2^{-1}] V_{z_1,z_2}(\phi_0,\phi_0)\txi^0(z_1,X_1)\txi^0(z_2,X_2).\label{eqn:free-energy-localization-annulus}
\end{align}

\begin{remark}
As discussed in \cite[Remark 3.6]{FLZ16}, the restriction to $\fQ=1$ is well-defined.
\end{remark}

\begin{figure}
\vspace*{-1cm}
\definecolor{uuuuuu}{rgb}{0.26666666666666666,0.26666666666666666,0.26666666666666666}
\definecolor{ttqqqq}{rgb}{0.2,0.,0.}
\begin{tikzpicture}[line cap=round,line join=round,>=triangle 45,x=2.5cm,y=2.5cm]
\clip(-1.8301416483150117,-1.8874294848107016) rectangle (3.726718929896075,1.395033885226987);
\draw [line width=0.4pt] (0.,0.75)-- (0.,0.);
\draw [line width=0.4pt] (0.,0.)-- (1.,0.);
\draw [line width=0.4pt] (0.,0.)-- (-0.75,-0.75);
\draw [line width=0.4pt,color=uuuuuu] (1.,0.)-- (1.,-0.75);
\draw [line width=0.4pt] (1.,0.)-- (1.75,0.75);
\draw [shift={(-0.2735412203034273,-0.13121773832306596)},line width=0.4pt]  plot[domain=0.4472761678863026:1.577824695489095,variable=\t]({1.*0.30338571827907257*cos(\t r)+0.*0.30338571827907257*sin(\t r)},{0.*0.30338571827907257*cos(\t r)+1.*0.30338571827907257*sin(\t r)});
\draw [shift={(-0.23269130046039038,0.26750923406962407)},line width=0.4pt]  plot[domain=4.447645982957459:5.428291291479409,variable=\t]({1.*0.35455102823495604*cos(\t r)+0.*0.35455102823495604*sin(\t r)},{0.*0.35455102823495604*cos(\t r)+1.*0.35455102823495604*sin(\t r)});
\draw [rotate around={78.88126519061447:(-0.3002919119259589,0.047645041375420485)},line width=0.4pt] (-0.3002919119259589,0.047645041375420485) ellipse (0.31208738823402316cm and 0.2194515517918163cm);
\begin{scriptsize}
\draw [fill=ttqqqq] (1.,0.) circle (1.5pt);
\draw[color=ttqqqq] (1.1607588324034996,0.0338496178538174) node {$\fp_1$};
\draw [fill=uuuuuu] (0.,0.) circle (1.0pt);
\draw[color=uuuuuu] (0.09023795843783826,0.062053836480717564) node {$\fp_0$};
\end{scriptsize}
\end{tikzpicture}
\vspace*{-3.5cm}
\caption{A $\T$-invariant holomorphic disk: it contains $\fp_0$ at the origin, and its boundary maps onto the intersection of $\L$ and the fiber at $\fp_0$ in $X=\cO_{\bP^1}(-1)\oplus \cO_{\bP^1}(-1)$.}
\end{figure}
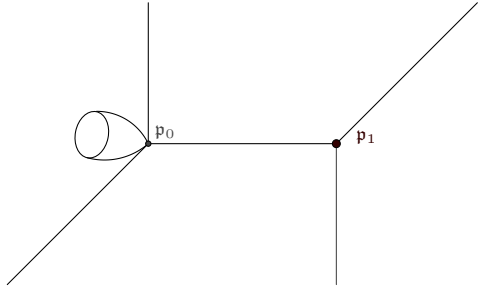

\begin{proposition}
\begin{equation}
F^{\cX,\L}_{g,n}(\btau; X_1,\dots,X_n)=\sum_{\vGa\in \bGa_{g,n}(\cX)} \frac{w^X_A(\vGa)}{\Aut(\vGa)}.\label{eqn:graph-open}
\end{equation}
Here $X=(X_1,\dots,X_n)$ and $w^X_A(\vGa)$ is obtain from $w^\bu_A(\vGa)$ by replacing the ordinary leaf by the \emph{open leaf} $(\cL^X)^\si_k(l_j)=[z^k](\txi^0(z,X_j) S_z(\hat \phi_\rho(\btau),\phi_0))_+ R(-z)_\rho^{\ \si}$.
\end{proposition}
\begin{proof}
From the graph sum formula \eqref{eqn:A-graph-sum} and Equations \eqref{eqn:free-energy-localization-stable}\eqref{eqn:free-energy-localization-disk}\eqref{eqn:free-energy-localization-annulus} which express open amplitude in terms of descendants, we expand $F^{\cX,\L}_{g,n}$ in terms of a graph sum. Notice the only difference is that we need to restrict to $t=\btau$ and then replaces the ordinary leaf $(\cL^\bu)^\beta_k(l_j)$ by the open leaf
\begin{align*}
(\cL^X)^\si_k(l_j)=&\sum_{a_i\geq 0} [z^k] (z^{a_i} S_z(\hat \phi_\rho(\btau),\phi_0))_+R_\rho^{\ \si}(-z) \Phi^{a_i}(X_j)\\=&[z^k](\txi^0(z,X_j) S_z(\hat \phi_\rho(\btau),\phi_0))_+ R(-z)_\rho^{\ \si}.
\end{align*}
\end{proof}

\subsection{Open-closed Gromov-Witten invariants as relative Gromov-Witten invariants I: the degree zero case}\label{relative1}
In this section, we interpret the open-closed Gromov-Witten invariants in terms of the relative Gromov-Witten invariants.

Recall that the curve class $u_*[\Si]=d+(\sum_{i=1}^h \mu_i)\gamma_0$. When $d=0$, the open-closed Gromov-Witten invariant $\langle \rangle^{\cX,L_{P,Q},\T}_{g,0,0,\vec{\mu}}$ can be expressed as the following relative Gromov-Witten invariant.

We consider the weighted projective plane
\begin{equation}
\bP(P,Q,1):=(\bC^3\setminus\{0\})/\bC^*,
\end{equation}
where the $\bC^*$ acts on $(\bC^3\setminus\{0\})$ by
\begin{equation}
t(x_1,x_2,x_3)=(t^Px_1,t^Qx_2,tx_3).
\end{equation}
Let $D_i=\{x_i=0\}\subset \bP(P,Q,1),i=1,2,3$ be the $\bC^*$-equivariant divisors. Then the canonical divisor $K$ of $\bP(P,Q,1)$ is equal to $D_1+D_2+D_3$. Let $\fX=\bP(P,Q,1)$. Then the Picard group $\Pic(\fX)$ is isomorphic to $\bZ[D_3]$ and we have $[D_1]=P[D_3]$ and $[D_2]=Q[D_3]$.

\begin{figure}
\hspace*{-1.5cm}
\begin{tikzpicture}[line cap=round,line join=round,>=triangle
  45,x=0.8cm,y=0.8cm]
\clip(-6.233889702008898,-5.675682738606269) rectangle (12.060751429340735,5.131046805249403);
\draw (0.,3.)-- (0.,0.);
\draw (0.,0.)-- (5.,0.);
\draw (5.,0.)-- (0.,3.);
\draw (0.,0.)-- (-2.,-2.);
\draw [->] (0.,3.) -- (1.2041544999043605,2.2775073000573833);
\draw [->] (5.,0.) -- (3.7163550706549024,0.7701869576070585);
\draw [->] (0.,0.) -- (0.,1.);
\draw [->] (0.,0.) -- (1.,0.);
\draw [->] (0.,3.) -- (0.,2.);
\draw [->] (5.,0.) -- (3.86343731813683,0.);
\draw (5.,0.)-- (7.5,-1.5);
\draw (0.,3.)-- (-1.7751253661628148,4.025979259433816);
\draw [->] (0.,3.) -- (-1.1091808009171789,3.6410795081832195);
\draw [->] (5.,0.) -- (6.211441696264751,-0.7268650177588509);
\draw [->] (0.,0.) -- (-1.1176452774519008,-1.1176452774519008);
\begin{footnotesize}
\draw (-3.6270751461281234,3.115070608616428) node[anchor=north west] {$P_2=(0,1,0)\cong \mathrm{pt}/\mathbb Z_Q$};
\draw (0.00003803435391537338,-0.08032023366576904) node[anchor=north west] {$P_3=(0,0,1)\cong \mathrm{pt}$};
\draw (4.997977679231994,0.673956574598269) node[anchor=north west]
{$P_1=(1,0,0)=\mathrm{pt}/\mathbb Z_P$};
\draw (3.036857977745489,2.6625045236580056) node[anchor=north west] {$D_3$};
\draw (-1.4487915369020774,1.1402367833433107) node[anchor=north west] {$D_1$};
\draw (1.8163009607364056,-1.5065891074741318) node[anchor=north west] {$D_2$};
\end{footnotesize}
\begin{scriptsize}
\draw [fill=black] (0.,3.) circle (1.5pt);
\draw [fill=black] (0.,0.) circle (1.5pt);
\draw [fill=black] (5.,0.) circle (1.5pt);
\draw[color=black] (1.2391710578068922,2.8505011231770015) node {$ u_1-\frac{P}{Q} u_2$};
\draw[color=black] (4.504269222913712,0.9779577080919368) node {$ u_2-\frac{Q}{P} u_1$};
\draw[color=black] (0.23603095343437064,0.564243584305318) node {$ u_2$};
\draw[color=black] (0.9660252859660328,0.17167749934689519) node {$ u_1$};
\draw[color=black] (0.3077473442083247,2.456792666858722) node {$\frac{- u_2}{Q}$};
\draw[color=black] (3.806559633101764,-0.338935614617825) node {$-\frac{ u_1}{P}$};
\draw[color=black] (-0.4579531776542831,3.896775664453704) node
{$\frac{P}{Q} u_2- u_1$};
\draw[color=black] (5.127116650110848,-0.5466004424108106) node {$\frac{Q}{P} u_1 - u_2$};
\draw[color=black] (-1.5848017383450856,-0.7466004424108106) node {$- u_1 - u_2$};
\end{scriptsize}
\end{tikzpicture}
\vspace*{-3.5cm}
\caption{The toric diagram of $\cO(-D_1-D_2)\to \bP(P,Q,1)$, with orbifold points $P_2,P_3$ and weights of the torus action labeled}
\end{figure}
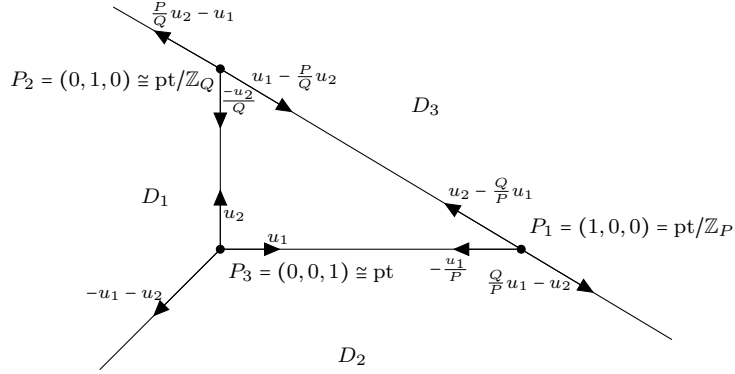

Consider the moduli space $\Mbar_g(\fX,D_3,d',\mu)$ of stable maps relative to the divisor $D_3$. Here $d'>0$ is an integer and $\mu=(\mu_1,\cdots,\mu_n)$ is a partition of $d'$. The degree of the stable maps is equal to $d'PQ[D_3]$. It is easy to see that the virtual dimension of $\Mbar_g(\fX,D_3,d',\mu)$ is equal to $(P+Q)d'+n+g-1$. Let
$$\pi:\cU_{g,\mu}\to \Mbar_g(\fX,D_3,d',\mu)$$
be the universal curve and let
$$\cT\to \Mbar_g(\fX,D_3,d',\mu)$$
be the universal target. Then there is a universal map
$$F:\cU_{g,\mu}\to\cT$$
and a contraction map
$$\tilde{\pi}:\cT\to\fX.$$
Define $\tF:=\tilde{\pi}\circ F$. The we define the obstruction bundle $V_{g,\mu}$ to be
$$V_{g,\mu}:=R^1\pi_*\tF^*\cO(-D_1-D_2).$$
By Riemann-Roch theorem, the rank of the obstruction bundle $V_{g,\mu}$ is equal to $(P+Q)d'+g-1$. Define
\begin{equation}\label{eqn:relative}
N_{g,\mu}:=\int_{[\Mbar_g(\fX,D_3,d',\mu)]^\vir}e(V_{g,\mu})\prod_{i=1}^n
\ev_i^*([\pt]),
\end{equation}
where $[\pt]$ is the point class on $D_3$. Then $N_{g,\mu}$ is a topological invariant.

We compute $N_{g,\mu}$ by standard localization computation. Consider the embedded 2-torus $T\cong(\bC^*)^2\subset \fX$. The $T-$action on itself extends to a $T-$action on $\fX$. Moreover, the $T-$action on $\fX$ lifts to a $T-$action on the line bundle $\cO(-D_1-D_2)$. Let $H^*_T(\pt;\bC)=\bC[u_1,u_2]$ and let $p_1=[1,0,0],p_2=[0,1,0],p_3=[0,0,1]$. Then the weights of the $T-$action at $p_1,p_2,p_3$ are given by
$$
\begin{array}{ccc}
      &        T_{\fX}            & \cO(-D_1-D_2) \\
  p_1 &   -\frac{u_1}{P},u_2-\frac{Qu_1}{P}        &  -u_2+\frac{Qu_1}{P}  \\
  p_2 & u_1-\frac{Pu_2}{Q},-\frac{u_2}{Q}         &   -u_1+\frac{Pu_2}{Q}        \\
  p_3 &  u_1, u_2      &   -u_1-u_2
\end{array}
$$

Let $u_1=Pu$ and $u_2=Qu$, where $u$ is the equivariant parameter of the corresponding sub-torus $T'\cong\bC^*\subset T$. Then the weights of the $T'-$action at $p_1,p_2,p_3$ are given by
$$
\begin{array}{ccc}
      &        T_{\fX}            & \cO(-D_1-D_2) \\
  p_1 &   -u,0        &  0  \\
  p_2 & 0,-u         &   0        \\
  p_3 &  Pu, Qu      &   -Pu-Qu
\end{array}
$$

Let $D=\{z\in\bC\mid|z|\leq1\}$ be the disk. Consider the map
\begin{eqnarray}
g:D&\to&\fX\\
t&\mapsto& [t^P,t^Q,1].
\end{eqnarray}
Then $g$ extends to a map
\begin{eqnarray}
\tilde{g}:\bP^1&\to&\fX\\
\empty[x,y] &\mapsto& [x^P,x^Q,y].
\end{eqnarray}
Consider the point $p_4=[1,1,0]\in D_3$, which is equal to $\tilde{g}([1,0])$. Recall that in \eqref{eqn:relative}), we pull back the point class on $D_3$ by using the evaluation map of the boundary marked points. We now put this point at $p_4$.

The $T'$-action on $\fX$ induces an $T'$-action on $\Mbar_g(\fX,D_3,d',\mu)$ and the $T'$-action on $\cO(-D_1-D_2)$ induces an $T'$-action on $V_{g,\mu}$. A general point $[f:(C,x_1,\cdots,x_n)\to \fX[m]]\in\Mbar_g(\fX,D_3,d',\mu)^{T'}\cap\ev^{-1}_1(p_4)\cap \cdots\cap\ev^{-1}_n(p_4)$ can be described as follows. The target $\fX[m]$ is the union of $\fX$ with $m$ copies of the surface
\begin{equation}
\Delta(D_3):=\bP(N_{D_3/\fX}\oplus\cO_{D_3}),
\end{equation}
where $N_{D_3/\fX}$ is the normal bundle of $D_3$ in $\fX$. The map
\begin{equation}
\Delta(D_3)\to D_3
\end{equation}
is a projective line bundle. There are two distinct sections
\begin{equation}
D_3^0:=\bP(N_{D_3/\fX}\oplus0),D_3^\infty:=
\bP(0\oplus\cO_{D_3}).
\end{equation}
The first copy of $\Delta(D_3)$ is glued to $\fX$ along $D_3^0$ and  $D_3$ and the $i-th$ copy of $\Delta(D_3)$ is glued to the $(i+1)-th$ copy of $\Delta(D_3)$ along $D_3^\infty$ and $D_3^0$. The $\fX$ component in $\fX[m]$ is called the root component and the union of the $m$ copies of $\Delta(D_3)$ is called the bubble component. The domain curve $C$ is decomposed as
\begin{equation}
C=C_0\cup C_1\cup\cdots\cup C_{l(\nu)}\cup C_\infty.
\end{equation}
Here $C_0$ is a possibly disconnected curve contracted to $p_3$, $\nu=(\nu_1,\cdots,\nu_{l(\nu)})$ is a partition of $d'$, for $i=1,\cdots,l(\nu)$, $C_i\cong\bP^1$ and $f\mid_{C_i}$ is a degree $\nu_i$ map to a line joining $p_3$ and a point in $D_3$ and the map is given by $z\mapsto z^{\nu_i}$, $C_\infty$ is a possibly disconnected curve mapping to the bubble component with ramification profile $\mu$ and $\nu$ over $D_3^\infty$ in the $m-$th copy of $\Delta(D_3)$ and over $D_3^0$ in the first copy of $\Delta(D_3)$ respectively, and $\tilde{f}:=\tilde{\pi}\circ f(x_i)=p_4$.

Since the $T'$-action along $D_3$ is trivial and the $T'$-action on $\cO(-D_1-D_2)\mid_{D_3}$ is also trivial, the contribution from the locus for $m>0$ vanishes. So we only need to consider the case when there is no bubble component. In this case, $C_0$ is a genus $g$ curve and $\nu=\mu$ and $C_i$ is mapped to the line joining $p_3$ and $p_4$, $i=1,\cdots,l(\nu)$. In this case, the tangent obstruction complex implies the following long exact sequence:
\begin{eqnarray*}
0&\to& \Aut(C,x_1,\cdots,x_n)\to H^0(C,f^*T\fX(-D_3))\to\cT^1\\
&\to &\Def(C,x_1,\cdots,x_n)\to H^1(C,f^*T\fX(-D_3))\to\cT^2\to0.
\end{eqnarray*}
So we have
\begin{eqnarray*}
\frac{1}{e_{T'}(N^\vir)}&=& \frac{e_{T'}(\cT^2)}{e_{T'}(\cT^1)}.
\end{eqnarray*}
Therefore
\begin{eqnarray*}
\frac{e_{T'}(V_{g,\mu})}{e_{T'}(N^\vir)}&=& \frac{e_{T'}(\cT^2)e_{T'}(V_{g,\mu})}{e_{T'}(\cT^1)}\\
&=&\Lambda^\vee_g(Pu)\Lambda^\vee_g(Qu)\Lambda^\vee_g(-(P+Q)u)\prod_{i=1}^{l(\mu)}
\mu_i\frac{(-1)^{(P+Q)\mu_i-1}((P+Q)\mu_i-1)!}
{(\frac{u}{\mu_i}-\psi_i)(P\mu_i)!(Q\mu_i)!}.
\end{eqnarray*}
Therefore, we have
\begin{eqnarray*}
N_{g,\mu}&=&\frac{1}{\prod_{i=1}^{n}\mu_i}\int_{\Mbar_{g,n}}
\frac{\Lambda^\vee_g(Pu)\Lambda^\vee_g(Qu)\Lambda^\vee_g(-(P+Q)u)}
{\frac{u}{\mu_i}-\psi_i}\prod_{i=1}^{l(\mu)}
\mu_i\frac{(-1)^{(P+Q)\mu_i-1}((P+Q)\mu_i-1)!}
{(P\mu_i)!(Q\mu_i)!}\\
&=&\int_{\Mbar_{g,n}}
\frac{\Lambda^\vee_g(Pu)\Lambda^\vee_g(Qu)\Lambda^\vee_g(-(P+Q)u)}
{\frac{u}{\mu_i}-\psi_i}\prod_{i=1}^{l(\mu)}
\frac{(-1)^{(P+Q)\mu_i-1}((P+Q)\mu_i-1)!}
{(P\mu_i)!(Q\mu_i)!}.\\
\end{eqnarray*}
We should notice that the factor $\frac{(-1)^{(P+Q)\mu_i-1}((P+Q)\mu_i-1)!}
{(P\mu_i)!(Q\mu_i)!}$ coincides with the disk factor $D(\mu_i)$ defined in the last Section. Therefore, we have proved the following theorem:

\begin{theorem}\label{thm:relative}
$$\langle \rangle^{\cX,L_{P,Q},\T}_{g,0,0,\vec{\mu}}=N_{g,\mu}.$$
\end{theorem}

In particular, when $g=0$ and $\mu=(d')$, we recover the disk factor $D(d')$ by the relative Gromov-Witten invariant $N_{0,(d')}$.

\begin{remark}\label{T}
One can also use the 2-dimensional torus $T\cong(\bC^*)^2$ to do the localization computation to calculate the relative Gromov-Witten invariant $N_{g,\mu}$. In this situation, one can put the point class insertion $[\pt]$ in the definition of $N_{g,\mu}$ at either $p_1$ or $p_2$ ($[\pt]=P[p_1]=Q[p_2]$) and we still obtain Theorem \ref{thm:relative}.
\end{remark}

\subsection{Open-closed Gromov-Witten invariants as relative Gromov-Witten invariants II: the general case}\label{relative2}
In the general case when $d>0$, we cannot express our open-closed Gromov-Witten invariants as relative Gromov-Witten invariants of an ordinary relative toric CY 3-fold. One way to solve this problem is to consider \emph{formal toric CY 3-folds} as in \cite{LLLZ09}, which belong to a larger class of target spaces. Then we can interpret the open-closed Gromov-Witten invariants $\langle \gamma_1,\dots,\gamma_m\rangle^{\cX,L_{P,Q},\T}_{g,n,d,\vec{\mu}}$ as relative Gromov-Witten invariants of a certain formal toric CY 3-fold. For completeness, we briefly review the construction of formal toric CY 3-folds in \cite{LLLZ09}.
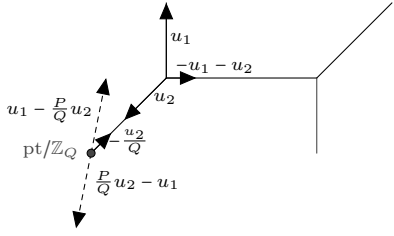
\begin{figure}
\vspace*{-2cm}
\hspace*{-0.1cm}
\definecolor{uuuuuu}{rgb}{0.26666666666666666,0.26666666666666666,0.26666666666666666}
\begin{tikzpicture}[line cap=round,line join=round,>=triangle 45,x=1.0cm,y=1.0cm]
\clip(-5.277270965464461,-4.322240478761797) rectangle (8.360643892847401,3.733739212654746);
\draw (0.,1.)-- (0.,0.);
\draw (0.,0.)-- (2.,0.);
\draw (0.,0.)-- (-1.,-1.);
\draw [->] (0.,0.) -- (0.,1.);
\draw [->] (0.,0.) -- (0.4,0.);
\draw [color=uuuuuu] (2.,0.)-- (2.,-1.);
\draw [->,line width=0.4pt] (0.,0.) -- (-0.5588226387259504,-0.5588226387259504);
\draw (2.,0.)-- (3.,1.);
\draw [->,line width=0.4pt,dash pattern=on 2pt off 2pt] (-1.,-1.) -- (-0.8,0.);
\draw [->,line width=0.4pt,dash pattern=on 2pt off 2pt] (-1.,-1.) -- (-1.2,-2.);
\draw [->,line width=0.4pt] (-1.,-1.) -- (-0.7308725784788683,-0.7308725784788683);
\begin{scriptsize}
\draw [fill=uuuuuu] (-1.,-1.) circle (1.5pt);
\draw[color=uuuuuu] (-1.5323580608298719,-0.9565334795356447) node {$\mathrm{pt}/\bZ_Q$};
\draw[color=black] (0.2008676852454157,0.5236153254659204) node {$u_1$};
\draw[color=black] (0.6466420143079925,0.15557564413216973) node {$- u_1- u_2$};
\draw[color=black] (-0.018146301262964,-0.27380398409053947) node {$u_2$};
\draw[color=black] (-1.5492811197561595,-0.4271538513129356) node {$ u_1-\frac{P}{Q} u_2$};
\draw[color=black] (-0.3972215977555294,-1.3983696770547778) node {$\frac{P}{Q} u_2 - u_1$};
\draw[color=black] (-0.5007947894594198,-0.854300234720714) node {$-\frac{u_2}{Q}$};
\end{scriptsize}
\end{tikzpicture}
\vspace*{-2.5cm}
\caption{Toric diagram of the formal toric Calab-Yau $3$-fold}
\label{fig:relative-ii}
\end{figure}

\subsubsection{Formal toric CY graphs}
Let $\Gamma$ be an oriented planar graphs. Let $E^o(\Gamma)$ denote the set of oriented edges of $\Gamma$ and let $V(\Gamma)$ denote the set of vertices of $\Gamma$. We define the \emph{orientation reversing map} $-:E^o(\Gamma)\to E^o(\Gamma), e\mapsto -e$ which reverses the orientation of an edge. Then the set of equivalent classes $E(\Gamma):=E^o(\Gamma)/\{\pm 1\}$ is the set of usual edges of $\Gamma$. There is also an \emph{initial vertex map} $\mathfrak{v}_0:E^o(\Gamma)\to V(\Gamma)$ and a \emph{terminal vertex map
}$\mathfrak{v}_1:E^o(\Gamma)\to V(\Gamma)$. We require that the orientation reversing map is fixed point free and that both $\mathfrak{v}_0$ and $\mathfrak{v}_1$ are surjective and $\mathfrak{v}_0(e)=\mathfrak{v}_1(-e)$ for $\forall e\in E^o(\Gamma)$. We also require that the valence of any vertex of $\Gamma$ is 3 or 1 and let $V_1(\Gamma)$ and $V_3(\Gamma)$ be the set of univalent vertices and trivalent vertices respectively. Notice that the absence of the bivalent vertices corresponds to the regular condition in \cite{LLLZ09}, which implies the corresponding formal toric CY 3-fold is smooth.

We define
$$
E^\mathfrak{f}=\{e\in E^o(\Gamma)|\mathfrak{v}_1(e)\in V_1(\Gamma)\}.
$$
We fix a standard basis $u_1,u_2$ of $\bZ^{\oplus 2}$ such that the ordered basis $(u_1,u_2)$ determines the orientation on $\bR^2$.

\begin{definition}
A formal toric CY graph is a planar graph $\Gamma$ described above together with a position map
$$
\mathfrak{p}:E^o(\Gamma)\to \bZ^{\oplus 2}\setminus \{0\}
$$
and a framing map
$$
\mathfrak{f}:E^\mathfrak{f}(\Gamma)\to \bZ^{\oplus 2}\setminus \{0\}
$$
such that
\begin{enumerate}
\item $\mathfrak{p}(-e)=-\mathfrak{p}(e)$ for $\forall e\in E^o(\Gamma)$.\\
\item  For a trivalent vertex $v\in V_3(\Gamma)$ with $\mathfrak{v}_0^{-1}(v)=\{e_1,e_2,e_3\}$, we have $\mathfrak{p}(e_1)+\mathfrak{p}(e_2)+\mathfrak{p}(e_3)=0$.\\
\item For a trivalent vertex $v\in V_3(\Gamma)$ with $\mathfrak{v}_0^{-1}(v)=\{e_1,e_2,e_3\}$, any two vectors in $\{\mathfrak{p}(e_1),\mathfrak{p}(e_2),\mathfrak{p}(e_3)\}$ form an integral basis of $\bZ^{\oplus 2}$.\\
\item For $e\in E^\mathfrak{f}(\Gamma), \mathfrak{p}(e)\wedge \mathfrak{f}(e)=u_1\wedge u_2$.

\end{enumerate}

\end{definition}

\subsubsection{Formal toric CY 3-folds}
The motivation of constructing the formal toric CY 3-folds is the following observation. Let $Y$ be a toric CY 3-fold and let $T\cong (\bC^*)^2$ be the CY torus acting on $Y$. Let $Y^1$ be the closure of the 1-dimensional $T-$orbits closure in $Y$. Then the $T-$equivariant Gromov-Witten theory of $Y$ is determined by $Y^1$ and the normal bundle of each irreducible component of $Y^1$ in $Y$. In other words, we only need the information of the formal neighborhood of $Y^1$ in $Y$ to study the $T-$equivariant Gromov-Witten theory of $Y$. So one may define $\hat{Y}$ to be the formal completion of $Y$ along $Y^1$ and try to define and study the Gromov-Witten theory of $\hat{Y}$. The advantage of considering $\hat{Y}$ is that if we are given a formal toric CY graph $\Gamma$, we can always construct an analogue of $\hat{Y}$ even if there does not exist a usual toric CY 3-fold corresponding to this formal toric CY graph. This construction is called the \emph{formal toric CY 3-fold} associated to $\Gamma$. The construction of formal toric CY 3-folds can be described as follows.

Let $\Gamma$ be a formal toric CY graph. For any edge $e\in E(\Gamma)$, we can associate a relative toric CY 3-fold $(Y_e,D_e)$, where $Y_e$ is the total space of the direct sum of two line bundles  over $\bP^1$ and $D_e$ is a divisor of $Y_e$ which can be empty or the fiber(s) of $Y_e\to \bP^1$ over 1 or 2 points on $\bP^1$, depending on the number of univalent vertices of $e$ (see \cite{LLLZ09}). Here $(Y_e,D_e)$ being a relative CY 3-fold means
$$\Lambda^3\Omega_{Y_e}(\log D_e)\cong\cO_{Y_e}.$$
Furthermore, the formal toric graph also determines a $T\cong (\bC^*)^2$ action on $Y_e$ and $D_e$ is a $T$-invariant divisor. Let $\Sigma(e)$ be the formal completion of $Y_e$ along its zero section $\bP^1$. The divisor $D_e$ descends to a divisor $\hat{D}_e$ of $\Sigma(e)$ and there is an induced $T$ action on $\Sigma(e)$. The formal relative toric CY 3-fold $(\hat{Y},\hat{D})$ is obtained by gluing all the $(\Sigma(e),\hat{D}_e)$'s along the trivalent vertices of $\Gamma$ (Each trivalent vertex $v$ corresponds to a formal scheme $\Spec(\bC[[x_1,x_2,x_3]])$, which can be naturally embedded into $\Sigma(e)$ for $e$ connected to $v$). As a set, $\hat{Y}$ is a union of $\bP^1$'s. Each connected component of the divisor $\hat{D}$ corresponds to a univalent vertex of $\Gamma$. The $T$ actions on each $\Sigma(e)$ are also glued together to form a $T$ action on $\hat{Y}$ such that the divisor $\hat{D}$ is $T$-invariant.

The result in \cite{LLLZ09} shows that one can define $T-$equivariant relative Gromov-Witten invariants for the formal relative toric CY 3-fold $(\hat{Y},\hat{D})$. When the formal toric CY graph comes from the toric graph of a usual toric CY 3-fold, these formal Gromov-Witten invariants coincide with the usual Gromov-Witten invariants of the toric CY 3-fold.

\subsubsection{Our case}
Now we apply the above construction to our case. First of all, we should notice that the above construction can be generalized to the case when there are orbifold structures on $\hat{D}$. This means that if we have an edge $e\in E^\mathfrak{f}(\Gamma)$, $Y_e$ can be the direct sum of two orbifold line bundles over the weighted projective line $\bP(1,Q)$ and the divisor $D_e$ is the fiber over the orbifold point on $\bP(1,Q)$.

Now we consider the following formal toric CY graph $\Gamma$ (see Figure \ref{fig:relative-ii}). There are two trivalent vertices $v_0$ and $v_1$ and an edge $e$ connecting them. There are four edges $e_1,e_2,e_3,e_4\in E^\mathfrak{f}(\Gamma)$ such that $\mathfrak{v}_0(e_1)=\mathfrak{v}_0(e_2)=v_0,
\mathfrak{v}_0(e_3)=\mathfrak{v}_0(e_4)=v_1$. The graph $\Gamma$ defines a formal relative toric CY 3-orbifold $(\hat{Y},\hat{D})$. The divisor $\hat{D}$ is of the form $\hat{D}_1\cup\hat{D}_2\cup\hat{D}_3\cup\hat{D}_4$, where each $\hat{D}_i$ is a connected component of $\hat D$ corresponding to the univalent vertex $\mathfrak{v}_1(e_i)$. The 3-fold $\hat{Y}$ is obtained by gluing $\Sigma(e),\Sigma(e_1),\Sigma(e_2)$ along the trivalent vertex $v_0$ and by gluing $\Sigma(e),\Sigma(e_3),\Sigma(e_4)$ along the trivalent vertex $v_1$. Here $\Sigma(e)$ is the formal completing of $Y_e$ along the zero section and $Y_e$ is the total space of $\cO_{\bP^1}(-1)\oplus\cO_{\bP^1}(-1)\to\bP^1$. Similarly, let $Y_{e_1}$ be the total space of $N_{D_1/\fX}\oplus \cO_{\fX}(-D_1-D_2)\mid_{D_1}\to D_1$ and then $\Sigma(e_1)$ is the formal completion of $Y_{e_1}$ along the zero section $\bP(1,Q)$. The $T$ action on $\cO_{\fX}(-D_1-D_2)\to\fX$ (see Section \ref{relative1}) induces a $T$ action on $Y_{e_1}$ which coincides with the $T$ action given by the formal toric CY graph.

\subsubsection{The formal relative Gromov-Witten invariants}
Let $\Mbar_{g,m}(\hat{Y},\hat{D}_1,d,\mu)$ be the moduli space of relative stable maps to $(\hat{Y},\hat{D}_1)$, where $\mu=(\mu_1,\cdots,\mu_n)$ is a partition of some positive integer $d'$ and $d$ is the degree of the stable map restricted to those components which are mapped to the $T-$invariant 1-orbit corresponding to the edge $e$. Define the formal relative Gromov-Witten invariant as
\begin{equation}
\langle \gamma_1,\dots,\gamma_m\mid\mu,\hat{D}_1\rangle^{\hat{Y},T}_{g,m,d,\mu}
:=\prod_{i=1}^{n}(u_1-\frac{P}{Q}u_2)
\int_{[\Mbar_{g,m}(\hat{Y},\hat{D}_1,d,\mu)^T]^{\vir}}
\frac{\prod_{j=1}^m \ev_j^*\gamma_j}{e_{T}(N^\vir)}.
\end{equation}
where $\gamma_1,\cdots,\gamma_m$ are $T-$equivariant cohomology classes of $\cX$. Notice that since the divisor $\hat{D}$ is decomposed into 4 connected components $\hat{D}_1\cup\hat{D}_2\cup\hat{D}_3\cup\hat{D}_4$, the above relative Gromov-Witten invariant is a special case of the relative Gromov-Witten invariant of $(\hat{Y},\hat{D})$ by letting the ramification profiles over $\hat{D}_2,\hat{D}_3,\hat{D}_4$ be empty. Therefore, the result in \cite{LLLZ09} shows that the above formal relative Gromov-Witten invariant is well-defined. We should also notice that since $\gamma_1,\cdots,\gamma_m$ are $T-$equivariant cohomology classes of $\cX$, they can be viewed as $T-$equivariant cohomology classes of $\hat{Y}$.

By Theorem \ref{thm:relative}, Remark \ref{T} and standard localization computation on $\langle \gamma_1,\dots,\gamma_m\mid\mu,\hat{D}_1\rangle^{\hat{Y},T}_{g,m,d,\mu}$, it is easy to obtain the following theorem

\begin{theorem}
The open-closed Gromov-Witten invariant $\langle \gamma_1,\dots,\gamma_m\rangle^{\cX,L_{P,Q},\T}_{g,m,d,\vec{\mu}}$ can be expressed as the formal relative Gromov-Witten invariant in the following way:
$$
\langle \gamma_1,\dots,\gamma_m\rangle^{\cX,L_{P,Q},\T}_{g,m,d,\vec{\mu}}=
\langle \gamma_1,\dots,\gamma_m\mid\mu,\hat{D}_1\rangle^{\hat{Y},T}_{g,m,d,\mu}
\mid_{u_1=Pu,u_2=Qu}
$$
\end{theorem}



\section{The spectral curve of a torus knot}
\label{sec:spectral}

\subsection{Mirror curve and the disk invariants}

The mirror curve to a resolved conifold
$\cX=[\cO_{\bP^1}(-1)\oplus\cO_{\bP^1}(-1)]$ is the following smooth affine
curve $C_q$
$$1+U+V+q U V=0$$
in $(\bC^*)^2$. This curve allows a compactification into a genus $0$ projective
curve $\overline C_q$ in $\bP^1\times \bP^1$, where $(1:U)$ and $(1:V)$ are
homogeneous coordinates for each $\bP^1$.

For coprime $P,Q\in \bZ$, consider the following change of variables
$$
X=U^Q V^P,\quad Y=U^\gamma V^\delta,
$$
where integers $\gamma,\delta$ are (not uniquely) chosen such that
$$
\begin{pmatrix}
Q & P\\
\gamma & \delta
\end{pmatrix}\in \mathrm{SL}(2;\bZ).
$$
We define the spectral curve as the quadruple
\[
(C_q\subset (\bC^*)^2, \overline C_q \subset \bP^1\times \bP^1, X, Y).
\]
The variables $X,Y$ are holomorphic functions on $C_q$ and meromorphic on $\overline C_q$. \footnote{Usually one says the mirror curve to $\cX$ is just $C_q$. This implicitly considers $\cX$ with Aganagic-Vafa branes (conifold transitions of an unknot) while choosing $U,V$ as the holomorphic functions.}

Conversely, $V=X^{-\gamma} Y^Q$. We let
\[
e^{-u}=U,\ e^{-v}=V,\ e^{-x}=X,\ e^{-y}=Y.
\]
The mirror curve equation can be rewritten into
$$
X=-V^P\left ( \frac{V+1}{1+q V}\right )^Q.
$$
We solve $v=-\log V(X)$ around $\mathfrak s_0=(X,V)=(0,-1)$ with $v|_{X=0}=-\sqrt{-1}\pi$. Let $\eta=X^{\frac{1}{Q}}$. Here $\eta$ is a local coordinate for the mirror curve $\overline C_q$ around $V=-1$. 
There exists $\delta>0$ and $\epsilon>0$ such that for $|q|<\epsilon$, the function $\eta$ is well-defined and restricts to an isomorphism
\[\eta: D_q \to D_\delta=\{\eta\in \bC:|\eta|<\delta\},\]
where $D_q\subset \overline C_q$ is an open neighborhood of $\mathfrak s_0$. Denote the inverse map of $\eta$ by $\rho_q$ and
\[
\rho^{\times n}_q=\rho_q\times \dots \times \rho_q : (D_\delta)^n \to (D_q)^n\subset (\overline C_q)^n.
\]

By the Lagrange inversion theorem, we have
$$
v(\eta)=-\sqrt{-1}\pi -\sum_{w> 0, 0\leq d \leq w} w B_{0,1}(w,d)\eta^w q^d.
$$
where
\begin{align}
B_{0,1}(w,d)=\frac{1}{w}\frac{e^{-\pi \sqrt{-1}(\frac{(P+Q)w}{Q} - d -1)}}{\Gamma(d+1)
  \Gamma(w-d+1)} \prod_{m=   1}^{w -1} (m-d+\frac{Pw}{Q}).\label{eqn:B-disk}
\end{align}
We define
\[
W_{0,1}(\eta,q) =-\int_{\eta'=0}^{\eta'=\eta} \rho_q^* (v(\eta')-v(0))\frac{d\eta'}{\eta'}= \sum_{w\geq 0, 0\leq d \leq w} B_{0,1}(w,d) \eta^w q^d.
\]
The numbers $B_{0,1}(w,d)$ are called \emph{ B-model disk invariants}, and
$W_{0,1}(\eta,q)$ is the \emph{B-model disk amplitude}.

\subsection{Differential forms on a spectral curve}

Eynard-Orantin's topological recursion \cite{EO07} is a recursive
algorithm which produces higher genus B-model invariants. It needs the
following input data
\begin{itemize}
\item
The affine curve $C_q$, its compactification $\overline C_q$ and meromorphic functions $X,Y$ on $\overline C_q$ which are holomorphic on $C_q$.
\item
A fundamental bidifferential $\omega_{0,2}$ on $(\overline C_q)^2$
$$
\omega_{0,2}=\frac{dU_1dU_2}{(U_1-U_2)^2}.
$$The choice of $\omega_{0,2}$ involves a symplectic basis on $H_1(\overline C_q;\bC)$ -- since the genus of our $\overline C_q$ is $0$, this extra piece of datum is not needed.
\end{itemize}

There are two ramification points of the map $X: C_q\to \bC$, labeled by $P_0=(U_0(q),V_0(q)),P_1=(U_1(q),V_1(q))$ for $\si=0,1$. They are given in $(U,V)$ coordinates below
\begin{align*}
U_0(q)&=-\frac{-\sqrt{(P (-q)+P+q Q+Q)^2-4 q Q^2}+P (-q)+P+q Q+Q}{2 q Q},\\
V_0(q)&=-\frac{-\sqrt{(P (-q)+P+q Q+Q)^2-4 q Q^2}+P q+P-q Q+Q}{2 P q},\\
U_1(q)&=-\frac{\sqrt{(P (-q)+P+q Q+Q)^2-4 q Q^2}+P (-q)+P+q Q+Q}{2 q Q},\\
V_1(q)&=\frac{\sqrt{(P (-q)+P+q Q+Q)^2-4 q Q^2}+P q+P-q Q+Q}{2 P q}.
\end{align*}
In particular,
\[
U_0(0)=-\frac{Q}{P+Q},\quad V_0(0)=\frac{-P}{P+Q}.
\]
 Let $e^{-x}=X,e^{-y}=Y, e^{-u}=u, e^{-v}=V$. Near each ramification point $P_\si$ with $x(P_\si)=x_\si$, $y(P_\si)=y_\si$, we expand
\begin{align*}
&x=x_\si+\zeta_\si^2,\\
&y=y_\si+\sum_{k\in\bZ_{\geq 1}}h^\si_k (\zeta_\si)^k.
\end{align*}

We expand the fundamental bidifferential
\[
\omega_{0,2}=\left (\frac{\delta_{\si\si'}}{(\zeta_\si-\zeta_{\si'})^2}+\sum_{k,l\in \bZ_{\geq 0}} B^{\si,\si'}_{k,l} \zeta_\si^k \zeta_{\si'}^l\right )d\zeta_\si d\zeta_{\si'},
\]
and define
\begin{align*}
\check B^{\si,\si'}_{k,l} &= \frac{(2k-1)!!(2l-1)!!}{2^{k+l+1}}B^{\si,\si'}_{k,l},\\
\check h^\si_k &= 2(2k-1)!! h^\si_{2k-1}.
\end{align*}
We also define the differential of the second kind \cite{EO07}.
\[
\theta^d_\si(p)=-(2d-1)!!2^{-d}\Res_{p'\to p_\si} B(p,p')\zeta_\si^{-2d-1}.
\]
They satisfy and are uniquely characterized by the following properties.
\begin{itemize}
\item $\theta^d_\si$ is a meromorphic $1$-form on $\overline C_q$ with a single pole of order $2d+2$ at $P_\si$.
\item In local coordinates
\[
\theta^d_\si=\left(-\frac{(2d+1)!!}{2^d\zeta_d^{2d+2}}+\text{analytic part in $\zeta_\si$}\right).
\]
\end{itemize}

In the asymptotic expansion
we define the formal power series by the asymptotic expansion
\begin{align*}
\check R_{\si'}^{\ \si}(z)&=\frac{\sqrt{z}e^{-\frac{\check  u^\si}{z}}}{2\sqrt{\pi}}
\int_{\gamma_\si}e^{-\frac{ x}{z}}\theta^0_{\si'}.
\end{align*}
Here $\gamma_\alpha$ is the Lefschetz thimble under the map $x$, i.e. $x(\gamma_\alpha)=[x_\alpha,\infty)$.

For $\si=0,1$, define
\begin{align*}
&\hat \xi^k_{\si}=(-1)^k(\frac{d}{d x})^{k-1} \frac{\theta_{\si}^0}{d x},\ k\in \bZ_{\geq 1},\\
&\htheta^k_\si=d\hxi^k_\si,\ k\geq 1,\quad \htheta^0_\si=\theta^0_\si,\\
&\theta_\si(z)=\sum_{k=0}^\infty \theta_\si^k z^k,\quad \htheta_\si(z)=\sum_{k=0}^\infty \htheta^k_\si z^k.
\end{align*}
We have the following proposition from \cite[Proposition 6.5]{FLZ16}.
\begin{proposition}
\label{prop:b-open-leaf}
\[
\theta_\si(z)=\sum_{\si'=0}^1 \check R_{\si'}^{\ \si}(z)\htheta_{\si'}(z).
\]
\end{proposition}

Define the following meromorphic form on $(\overline C_q)^2$
\[
C(p_1,p_2)=-(\frac{\partial}{\partial x(p_1)}-\frac{\partial}{\partial
  x(p_2)} )\left ( \frac{\omega_{0,2}}{dx(p_1) dx(p_2)}\right
)(p_1,p_2) d x(p_1) d x(p_2).
\]
It is holomorphic on $(\overline C_q)^2\setminus
\{P_\si:\si=0,1\}$. Lemma 6.8 of \cite{FLZ16} says
\begin{equation}
\label{eqn:C}
C(p_1,p_2)=\frac{1}{2}\sum_{\si=0}^1 \theta^0_\si(p_1) \theta^0_\si(p_2).
\end{equation}

\subsection{Higher genus invariants from Eynard-Orantin's topological recursion}
\label{sec:EO-recursion}

For a point $p$ around a ramification point $P_\alpha$, denote $\bar p$
to be the point such that $X(\bar p)=X(p)$, and $\bar p\neq p$ (i.e. $\zeta_\si (p)=-\zeta_\si(\bar p)$). The Eynard-Orantin recursion is the following.
\begin{definition}
\begin{align*}
\omega_{g,n+1}(p_1,\dots, p_{n+1})=&\sum_{\alpha=0}^1\Res_{p\to P_\alpha}\frac{-\int_{\xi=\bar p}^{p}
  \omega_{0,2}(p_{n+1},p)}{2(\Phi(p)- \Phi(\bar
  p))}\cdot(\omega_{g-1,n+2}(p,\bar p, p_1,\dots, p_n)\\
&+\sum'_{g_1+g_2=g, I\sqcup J=\{1,\dots, n\}}
  \omega_{g_1,|I|}(p_I, p) \omega_{g_2,|J|} (p_J,\bar p)),
\end{align*}
where $\Phi=\log Y\frac{dX}{X}$, and $\sum'$ excludes the case
$(g_1,|I|)=(0,0), (0,1), (g,n-1), (g,n)$.
\end{definition}
As shown in \cite{EO07}, these $\omega_{g,n}$ are symmetric meromorphic forms on $(\overline C_q)^n$, and they are smooth on $(\overline C_q\setminus \{P_0, P_1\})^n$. We define the B-model open potentials as below
\begin{align*}
W_{0,2}(\eta_1,\eta_2,q)&=\int_{0}^{\eta_1}\int_0^{\eta_2}
  (\rho_q^{\times 2})^* C,\\
W_{g,n}(\eta_1,\dots,\eta_n,q)&=\int_{0}^{\eta_1}\dots \int^{\eta_n}
                              (\rho_q^{\times n})^*\omega_{g,n},\ 2g-2+n>0.
\end{align*}

The Eynard-Orantin topological recursion expresses the B-model
invariants $\omega_{g,n}$ as graph sums \cite{DOSS}. For each
decorated graph $\vGa\in \bGa_{g,n}$ and $\bp=(p_1,\dots,p_n)\in (\overline C_q)^n$, we assign the following weights to each graph component.

\begin{itemize}
\item Vertex: for a vertex $v\in V(\Ga)$ labeled by $\si=\beta(v)$ and
  $g(v)$, the weight is $$\left (\frac{h_1^\si}{\sqrt{-2}}\right)^{2-2g-\val(v)}\langle \prod_{h\in H(v)} \tau_{k(h)}\rangle_{g(v)}.$$
\item Edge: for an edge $e\in E(\Ga)$ we assign the weight $\check B^{\beta(v_1(e)),\beta(v_2(e))}_{k,l}$. In \cite{EO07}, it is known to be equal to
\[
[z^kw^l]  \left(\frac{1}{z+w}(\delta_{\beta(v_1(e)),\beta(v_2(e))} - \check R(z)_{\gamma}^{\ \beta(v_1(e))}\check R(w)_{\gamma}^{\ \beta(v_2(e))} )\right).
\]
\item Dilaton leaf: for a dilaton leaf $l\in \cL^1(\Ga)$ we assign the weight
$$(\check{\mathcal L^1})^{\beta(l)}_{k(l)}=( -\frac{1}{\sqrt{-2}})[z^{k(l)-1}]\sum_{\si=0}^1 h_1^\si \check R_\si^{\ \beta(l)}.$$
\item Ordinary leaf: for the $j$-th ordinary leaf $l_j$ we assign the weight
\begin{align*}
(\check\cL^\bp)^{\beta(l_j)}_{k(l_j)}(l_j)=-\frac{1}{\sqrt{-2}} \theta^{k(l_j)}_{\beta(l_j)}(p_j).
\end{align*}
\end{itemize}
We define the ordinary B-model weight of a decorated graph to be
\begin{align*}
w^\bp_B(\vGa)=&(-1)^{g(\vGa)-1}\prod_{v\in V(\Ga)}\left (\frac{h_1^{\beta(v)}}{\sqrt{-2}}\right)^{2-2g-\val(v)}\langle \prod_{h\in H(v)} \tau_{k(h)}\rangle_{g(v)}\cdot \prod_{e\in E(\Ga)}\check B^{\beta(v_1(e)),\beta(v_2(e))}_{k,l}\\
&\cdot \prod_{l\in \cL^1(\Ga)}(\check{\mathcal L^1})^{\beta(l)}_{k(l)}\cdot \prod_{j=1}^n (\check\cL^\bp)^{\beta(l_j)}_{k(l_j)}(l_j).
\end{align*}
For $\eta=(\eta_1,\dots,\eta_n)$, one defines the B-model open leaf weight associated to an ordinary leaf $l_j$
$$(\check{\cL}^{\eta}_j)^{\beta(l_j)}_{k(l_j)} (l_j) = -\frac{1}{\sqrt{-2}} \int_0^{\eta_j} \rho_q^*\theta_{\bsi}^k.
$$
We define the B-model open weight by substituting the ordinary leaf term by the open leaf
\begin{align}
  \label{eqn:B-graph-weight}
w^\eta_B(\vGa)=&(-1)^{g(\vGa)-1}\prod_{v\in V(\Ga)}\left (\frac{h_1^{\beta(v)}}{\sqrt{-2}}\right)^{2-2g-\val(v)}\langle \prod_{h\in H(v)} \tau_{k(h)}\rangle_{g(v)}\cdot \prod_{e\in E(\Ga)}\check B^{\beta(v_1(e)),\beta(v_2(e))}_{k,l}\\
&\cdot \prod_{l\in \cL^1(\Ga)}(\check{\mathcal L^1})^{\beta(l)}_{k(l)}\cdot \prod_{j=1}^n (\check\cL^\eta)^{\beta(l_j)}_{k(l_j)}(l_j).\nonumber
\end{align}

\begin{theorem}[Graph sum formula for $\omega_{g,n}$ \cite{Ey11, DOSS}]
\label{thm:B-model-graph}
\begin{align}
\label{eqn:B-graph-sum}
\omega_{g,n}(\bp)&=\sum_{\vGa\in \bGa_{g,n}}\frac{w^\bp_B(\vGa)}{\Aut(\vGa)};\\
\nonumber\int_0^{\eta_1}\dots \int_0^{\eta_n} (\rho_q^{\times n})^*\omega_{g,n}&=\sum_{\vGa\in \bGa_{g,n}}\frac{w^\eta_B(\vGa)}{\Aut(\vGa)}.
\end{align}
\end{theorem}

\section{Mirror symmetry for open invariants with respect to $\L$}
\label{sec:mirror}
\subsection{Mirror theorem for genus $0$ descendants}
We follow the notation of Givental \cite{Gi97}. Define the
equivariant small I-function as below
\begin{equation}
I(z,\tau_0,\tau_1)=e^{\frac{\tau_0 \sv+\tau_1 \HT}{z}}\sum_{d\geq 0} e^{\tau_1 d} \frac{\prod_{m=-d}^{-1}(\frac{D_1}{z}-d-m)(\frac{D_2}{z}-d-m)}{\prod_{m=0}^{d-1}(\frac{D_3}{z}+d-m)\prod_{m=0}^{d-1}(\frac{D^4}{z}+d-m)}.\label{eqn:I}
\end{equation}
Here each $D_i$ is the equivariant first Chern class of the equivariant line bundle associated to each toric divisor. The equivariant small J-function is
\begin{align*}
J(z,\tau_0,\tau_1)&=e^{\frac{\tau_0 \sv+\tau_1 \HT}{z}}(1+\sum_{d>0} e^{\tau_1d} (\ev_1^d)_*
  \frac{1}{z(z-\psi_1)})\\
&=e^{\frac{\tau_0\sv+\tau_1 \HT}{z}} \sum_{\si=0}^1\langle \frac{\phi^\si }{z(z-\psi_1)}\rangle^{X}_{0,1,d}e^{\tau_1d}\phi_\si.
\end{align*}

We quote the famous mirror theorem \cite{Gi96a, Gi96b, Gi97, LLY97, LLY99a} below for $\cX=[\cO_{\bP^1}(-1)\oplus
\cO_{\bP^1}(-1)]$. The mirror map is trivial in this particular situation.
\begin{theorem}[Givental, Lian-Liu-Yau]
$$I(z,\tau_0,\tau_1)=J(z,\tau_0,\tau_1).$$
\label{thm:IJ}
\end{theorem}

In the rest of this paper, we denote the trivial mirror map as below
\begin{equation}
\label{eqn:mirror-map}
\tau_0=0,\quad \tau_1=\log q,\quad \btau=\btau(q)=(\log q) \sH^\T.
\end{equation}

For any $\btau\in H^*_\T(X;\bC)$, the canonical basis $\phi_\si(\btau)$ is decomposed as
\[
\phi_\si(\btau)=B_\si(\btau) \HT + C_\si(\btau),
\]
where $B_\si(\btau), C_\si(\btau)\in \bStt$. We quote a proposition
from \cite[Section 4.4 and Proposition 6.2, ]{FLZ16}.
\begin{proposition}
\label{prop:canonical-in-flat}
We have
\[
\frac{h^\si_1}{2}\theta^0_\si= -B_\si(\btau(q))\vert_{\sv=1} d\left ( \frac{q \frac{\partial \Phi}{\partial q}}{d x} \right )+C_\si(\btau(q))\vert_{\sv=1} d\left(\frac{d y}{d x}\right).
\]
\end{proposition}

\subsection{Mirror symmetry for disk invariants}

In this section we use the genus $0$ mirror theorem  to compute the descendant part of the disk potential, and
identify it with the non-fractional parts of the B-model disk
invariants.

Let $f\in \bC \llbracket \eta_1,\dots,\eta_n\rrbracket$. For a fixed integer number $Q>0$, we denote
\[
\mathfrak h_{\eta_1,\dots,\eta_n}\cdot f(\eta_1,\dots,\eta_n)=\sum_{k_1,\dots,k_n=0}^{Q-1} \frac{f(a^{k_1}\eta_1,\dots,a^{k_n}\eta_n)}{Q^n},
\]
where $a$ is a primitive $Q$-th root of unity. This operation ``throws away'' all terms with degree not divisible by $Q$. If $f$ is an actual function analytic at $0$, we also use the same notation.

Recall that in Section \ref{sec:A-model}, we define the A-model disk amplitude as
$$
F^{\cX,\L}_{0,1}(X,\btau)=\sum_{\mu>0} \llangle \rrangle_{0,0,\mu}|_{\fQ=1}  X^\mu.
$$

We have
\[
\iota_0^* D_1=Q\sv,\quad \iota_0^* D_2=P\sv,\quad \iota_0^*
D_3=-(P+Q)\sv,\quad \iota_0^* D_4=0,
\]
and
\begin{equation}
\iota_0^* I(z,t_0,t_1)=e^{\frac{t_0}{z}}\sum_{d\geq 0} e^{t_1d} S^d \frac{\prod_{m=-d}^{-1}(\frac{Q\sv}{z}-d-m)(\frac{P\sv}{z}-d-m)}{\prod_{m=0}^{d-1}(\frac{-(P+Q)\sv}{z}+d-m)d!}.\label{eqn:pullback}
\end{equation}

The pullback along $\iota_0$ is
$$
\iota_0^* J(z,t_0,t_1)=e^{\frac{t_0}{z}}\sum_{d\geq 0} \langle
\frac{\phi^0}{z(z-\psi_1)} \rangle_{0,1,d}^{\cX}e^{t_1d}.
$$

By the mirror theorem and Equation \eqref{eqn:pullback}, when $\tau_0=0$
\begin{align*}
F^{\cX,\L}_{0,1}(X,\tau_1)&=\sum_{\mu>0} D(\mu) X^\mu \sum_{d\geq 0} e^{dt_1}\langle
  \frac{\phi^0}{\frac{\sv}{\mu}(\frac{\sv}{\mu}-\psi_1)}\rangle_{0,1,d}^\cX\\
&=\sum_{\mu>0} D(\mu) X^\mu \sum_{d\geq 0} q^d
  \frac{\prod_{m=-d}^{-1}(Q\mu-d-m)(P\mu-d-m)}{\prod_{m=0}^{d-1}(-(P+Q)\mu+d-m)d!}\\
&=\sum_{\mu>0,0\leq d\leq Q \mu}(-1)^{(Q+P)\mu-d-1} X^\mu q^d \frac{\prod_{m=1}^{Q\mu-1}(P\mu-d+m)}{d!(Q\mu-d)!}.
\end{align*}
Here we identify $e^{\tau_1}=q$. Comparing with Equation \eqref{eqn:B-disk}, we have the following disk theorem.
\begin{theorem}[Disk mirror theorem for $\L$]
\label{thm:disk-mirror-theorem}
$$
F^{\cX,\L}_{0,1}=Q(\mathfrak h_\eta \cdot W_{0,1})(\eta,q),
$$
where $(\mathfrak h_\eta\cdot W_{0,1})(\eta,q)=\mathfrak h_\eta \cdot \left ( \int_{\eta'=0}^{\eta'=\eta}\rho^*_q(v(\eta')-v(0))(-\frac{d\eta'}{\eta'})\right )$.
\end{theorem}

\subsection{Matching the graphs sums}
\label{sec:match-graph}

The localization formula \eqref{eqn:free-energy-localization-stable} in Section \ref{sec:A-model},
and Givental's graph sum formula \cite{Gi01a,Gi01b} express $F_{g,n}^{\cX,\L}$ as in Equation \eqref{eqn:graph-open}. As a special case of \cite[Theorem 7.1]{FLZ16}, we know under the trivial mirror map \eqref{eqn:mirror-map}
\begin{equation}
\label{eqn:R-identification}
R_{\si'}^{\ \si}(z)|_{\btau=(\log q) \H^\T, \mathfrak Q=1,\sv=1}=\check R_{\si'}^{\ \si} (-z).
\end{equation}
Also from \cite[Section 4.4 and Lemma 4.3]{FLZ16} and the fact
$h^\si_1=\sqrt{2/\frac{d^2x}{dy^2}}$ by direct calculation, we have
\[
\frac{h_1^\si}{\sqrt{-2}}=\sqrt{\frac{1}{\Delta^\si(\btau)\vert_{\sv=1}}}.
\]

Immediately from these facts and the graph sum formulae \eqref{eqn:B-graph-sum} and \eqref{eqn:A-graph-sum}, the weights in the graph sum match except for the open leafs. We will compare them in the next section.

\subsection{Matching open leaves}

We define
\[
U_z(\tau,X)= \sum_{\si'=0}^1 \txi^{\si'}
  (z,X)S(1,\phi_{\si'})\big|_{\mathfrak Q=1}.
\]
By Equation \eqref{eqn:free-energy-localization-disk},
\begin{align*}
F^{\cX,\L}_{0,1}(\btau, X) &= [z^{-2}] \sum_{\si'=0}^1 \txi^{\si'}(z,X)S(1,\phi_{\si'})|_{\mathfrak Q=1},\\
(X\frac{d}{dX})^2F^{\cX,\L}_{0,1}(\btau,X) &=[z^0] \sum_{\si'=0}^1 \txi^{\si'}(z,X)S(1,\phi_{\si'})|_{\mathfrak Q=1} \sv^2.
\end{align*}
Since $F^{\cX,\L}_{0,1}=Q(\mathfrak h_\eta \cdot W_{0,1})(\eta,q)$ (Theorem \ref{thm:disk-mirror-theorem}), as power series in $X$,
\begin{align}
  [z^0](\sv^2 U(z)(\tau,X))&=-(\frac{\eta d}{Q d\eta}) (\mathfrak h_\eta\cdot \rho^*_q)(v)\nonumber \\
\left(\sv^2 U(z)(\tau,X)\right)_{\geq -1}&=-\sum_{n\geq -1} (\frac{z}{\sv})^n (\frac{\eta d}{Qd\eta})^{n+1}(\mathfrak h_\eta \cdot \rho_q^*)(v).\label{eqn:z-disk}
\end{align}
Here $()_{\geq a}$ truncates the $z$-series, keeping terms whose power is greater or equal than $a$. We usually denote $()_{\geq 0}$ by $()_+$.

If we write the canonical basis by flat basis
\[
\hat \phi_{\si}(\btau(q))=\hat B_\si (q)  \HT+ \hat C_\si(q) 1,
\]
then
\[
\sum_{\si'=0}^1 \txi^{\si'} (z,X) S(\hat \phi_\si(\btau(q)), \phi_{\si'})\big|_{\mathfrak Q=1}=z\hat B_\si(q) \frac{\partial U}{\partial \tau_1}+ \hat C_\si(q) U.
\]
Here we use the fact that $S(\H^\L,\phi_{\si'})=z\frac{\partial S(1, \phi_{\si'})}{\partial \tau_1}$.
From Proposition \ref{prop:canonical-in-flat},
\begin{align}
\label{eqn:b-canonical}
\nonumber -\frac{\theta^0_{\si}}{\sqrt{-2}}&=-\hat B_{\si}(q)\vert_{\sv=1} d(\frac{q\frac{\partial \Phi}{\partial q}}{d x})
+ \hat C_\si(q)\vert_{\sv=1} d(\frac{dy}{dx})\\
&=\frac{1}{Q}\left(-\hat B_{\si}(q)\vert_{\sv=1} d(q\frac{\partial v}{\partial q})
+ \hat C_\si(q)\vert_{\sv=1} d(\frac{dv}{dx})\right).
\end{align}
By Equation \eqref{eqn:z-disk} and \eqref{eqn:b-canonical}, under the open-closed mirror map

\begin{align*}
\mathfrak h_\eta \cdot \int_0^\eta \frac{\rho^*_q \theta_\si^0}{\sqrt{-2}}&=-\frac{1}{Q} \left(-\hat B_\si(q)\vert_{\sv=1}(q\frac{\partial (\mathfrak h_\eta \cdot \rho^*_q v)}{\partial q}) +\hat C_\si(q)\vert_{\sv=1} \frac{d(\mathfrak h_\eta \cdot \rho^*_q v)}{dx}\right)\\
&=-\frac{1}{Q}\left(-\hat B_\si(q)\vert_{\sv=1}\frac{\partial (\mathfrak h_\eta\cdot \rho_q^* v)}{\partial \tau_1}-\hat C_\si(q)\vert_{\sv=1}(\frac{\eta d}{Qd\eta})(\mathfrak h_\eta\cdot \rho_q^* v))\right)\\
&=-\frac{1}{Q}\left(\hat B_\si(q) [z^{-1}]\frac{\partial U(z)(\tau,X)}{\partial \tau_1}+\hat C_\si(q)[z^0] U(z)(\tau,X)\right)_{\sv=1},
\end{align*}
Therefore
\begin{align}
  \label{eqn:xi-in-disk}
  \mathfrak h_\eta \cdot \int_0^\eta \frac{\rho^*_q \hat \theta_\si(z)}{\sqrt{-2}}&=\sum_{a\geq 0} z^a (X\frac{d}{dX})^a (\mathfrak h_\eta \cdot \int_0^\eta \frac{\rho^*_q \theta_\si^0}{\sqrt{-2}})\\
  &=-\frac{1}{Q}(+\hat B_\si(q)z \frac{\partial U_{\geq -1}}{\partial \tau_1}+\hat C_\si(q)U_{\geq 0})\nonumber\\
  &=-\frac{1}{Q}\left(\sum_{\si'=0}^1 \txi^{\si'}(z,X)
  S( \hat\phi_{\si}(\btau(q)),\phi_{\si'})\big|_{\mathfrak
    Q=1}\right)_+.\nonumber
\end{align}

\begin{theorem}[BEM-DSV conjecture for annulus
  invariants]
  \label{thm:main-theorem-annulus}
\[
F^{\cX,\L}_{0,2}(X_1,X_2,\btau)|_{\btau=\btau(q)}=-Q^2 \mathfrak h_{\eta_1,\eta_2}
\cdot W_{0,2}(\eta_1,\eta_2,q).
\]
\end{theorem}
\begin{proof}
\begin{align*}
& \mathfrak h_{\eta_1,\eta_2}\cdot\left ((\eta_1\frac{\partial}{Q\partial \eta_1} + \eta_2\frac{\partial}{Q\partial \eta_2})W_{0,2}(q;\eta_1,\eta_2)\right)\\
=& \mathfrak h_{\eta_1,\eta_2}\cdot \left (\int_0^{\eta_1}\int_0^{\eta_2}(\rho_q^{\times 2})^*C\right )=\frac{1}{2}\sum_{\si=0}^1 (\mathfrak h_{\eta_1} \cdot \int_0^{\eta_1}(\rho_q)^* \theta_\si^0)(\mathfrak h_{\eta_2}\cdot \int_0^{\eta_2}(\rho_q)^* \theta_\si^0 )\\
=&-\frac{1}{Q^2}[z_1^{-2}z_2^{-2}] \sum_{\si,\si',\si''=0}^1 \txi^{\si''}(z_1,X_1)
   \txi^{\si'}(z_2,X_2) S(\hat\phi_{\si}(\btau),\phi_{\si'})\big|_{\btau=\btau(q), \fQ=1}  S(\hat\phi_{\si}(\btau),\phi_{\si''})\big|_{\btau=\btau(q), \fQ=1}\\
=&-\frac{1}{Q^2} [z_1^{-2}z_2^{-2}](z_1+z_2) \sum_{\si',\si''=0}^1 V( \phi_{\si'},\phi_{\si''})\big|_{\btau=\btau(q), \fQ=1}\txi^{\si'}(z_1,X_1) \txi^{\si''}(z_2,X_2)\\
=&-\frac{1}{Q^2}[z_1^{-1}z_2^{-1}] (X_1\frac{\partial }{\partial
   X_1}+X_2\frac{\partial}{\partial X_2}) \sum_{\si',\si''=0}^1 V( \phi_{\si'},\phi_{\si''})\big|_{\btau=\btau(q), \fQ=1}\txi^{\si'}(z_1,X_1) \txi^{\si''}(z_2,X_2)\\
=& -\frac{1}{Q^2}(X_1\frac{\partial}{\partial X_1} + X_2\frac{\partial}{\partial X_2})F^{\cX,\L}_{0,2}(X_1,X_2,\btau).
\end{align*}
where the second equality follows from Equation \eqref{eqn:C}, the fourth
equality follows from Equation \eqref{eqn:xi-in-disk}, the fifth equality
is WDVV (Equation \eqref{eqn:two-in-one}), and the last equality follows from \eqref{eqn:free-energy-localization-annulus}.

\end{proof}
\begin{theorem}[BEM-DSM conjecture for $2g-2+n>0$]
  \label{thm:main-theorem-stable}
$$F^{\cX,\L}_{g,n}(X_1,\dots,X_n,\btau)|_{\btau=\btau(q)}=(-1)^{g-1+n} Q^n (\mathfrak h \cdot
W_{g,n})(\eta_1,\dots,\eta_n,q).$$
\end{theorem}
\begin{proof}

  As shown in Section \ref{sec:match-graph}, other than open leaf terms, the A-model graph weights match B-model graph weights under the trivial mirror map \eqref{eqn:mirror-map} and after setting the A-model equivariant parameter $\sv=1$.  By Equation \eqref{eqn:xi-in-disk}, the A-model open leaf is
\[
(\cL^X)^\si_k{\vert_{\sv=1}}=-Q[z^k]\frac{1}{\sqrt{-2}}\mathfrak h \cdot \left( \int^\eta_0 \sum_{\si'=0}^1
  \left (R_{\si'}^{\ \si}(-z)\right)|_{\btau=\btau(q),\fQ=1,\sv=1}\rho_q^* \hat\theta_{\si'}(z)\right),
\]
while the B-model open leaf is
\[
(\check \cL^\eta)^\si_k=[z^k]\frac{1}{\sqrt{-2}}\int_0^\eta \sum_{\si'=0}^1 R_{\si'}^{\ \si} (-z)\vert_{\btau=\btau(q),\fQ=1,\sv=1}\rho_q^* \hat \theta_{\si'}(z).
\]
So \[(\mathcal L^X)^\si_k\vert_{\sv=1}=-Q \mathfrak h\cdot(\check{\mathcal
    L}^\eta)^\si_k.\]
Therefore, for any decorated stable graph $\vGa$, since A-model open potential does not depend on the equivariant parameter, i.e. a scalar of cohomological degree 0 in the equivariant cohomology, then
\[w^X_A(\vGa)=w^X_A(\vGa)|_{\sv=1}=(-1)^{g-1+n}Q^n (\mathfrak h\cdot w_B^\eta(\vGa)).
\]
The factor $(-1)^{g-1}$ comes from the B-model graph sum formula \eqref{eqn:B-graph-weight}, while the factor
$(-Q)^n$ comes from the factor $-Q$ for open leaves.
\end{proof}

\end{document}